 \documentclass[10 pt]{article}
 \usepackage[latin1]{inputenc}
\usepackage{amsfonts,amssymb,amsmath,amsthm}
\usepackage{tikz}
\usepackage{fancybox}
 \usepackage{pstricks}
 \usepackage[headinclude,DIV13]{typearea}
 \usepackage{hyperref}
\usepackage{graphicx, psfrag}

 \usepackage{amsmath,amsfonts,amssymb}
 \usepackage{color}
\newtheorem{rem}{Remark}
\newtheorem{lem}{Lemma}[section]
\newtheorem{pro}{Proposition}
\newtheorem{defi}{Definition}[section]

\newtheorem{theo}{Theorem}

\newtheorem{ass}{Assumption}

\renewcommand{\P}{\mathbb{P}}
\newcommand{\R}{\mathbb{R}}

\newcommand{\E}{\mathbb{E}}

\newcommand{\N}{\mathbb{N}}
\newcommand{\Z}{\mathbb{Z}}

\newcommand{\eps}{\varepsilon}

\numberwithin{equation}{section}

\title{The interplay of two mutations in a population of varying size: a stochastic eco-evolutionary model for clonal interference}

\author{Sylvain Billiard
\thanks{Unité Evo-Eco-Paléo, UMR CNRS 8198, Université des Sciences et Technologies Lille, Cité scientifique,
59655 Villeneuve d'Ascq Cedex-France; E-mail:{\texttt{sylvain.billiard@univ-lille1.fr}}}, Charline Smadi
\thanks{Irstea, UR LISC, Laboratoire d'Ingénierie des Systèmes Complexes, 9 avenue Blaise Pascal-CS 20085, 63178 Aubière, France
 and Department of Statistics, University of Oxford, 1 South Park Road, Oxford OX1 3TG, UK; E-mail:{\texttt{charline.smadi@polytechnique.edu}}}}

\begin{document}
 
\maketitle
 
\begin{abstract}
Clonal interference, competition between multiple co-occurring beneficial mutations, has a major role in adaptation of asexual populations. 
We provide a simple individual based stochastic model of clonal interference taking into account a wide variety of competitive interactions
which can be found in nature. 
In particular, we relax the classical assumption of transitivity between the different mutations.
It allows us to predict genetic patterns, such as coexistence of several mutants or emergence of Rock-Paper-Scissors cycles, which were not explained by
existing models. In addition, we call into questions some classical preconceived ideas about fixation time and fixation probability of competing mutations.

\end{abstract}



% \numberwithin{equation}{section}

% \tableofcontents

\section*{Introduction}

From the works of the 'great trinity' of Fisher \cite{fisher1922dominance,fisher1931evolution}, 
Wright \cite{wright1931evolution} and Haldane \cite{haldane1927mathematical} the questions of fixation probability and fixation time of new 
beneficial mutations 
have been widely studied.
These are indeed fundamental questions if we aim at understanding how and how fast a population can adapt to a changing environment, 
the dynamics of genetic diversity, or the long term behaviour of ecological systems.

The first models of adaptation in asexual populations postulated that the beneficial mutations were rare enough for populations to evolve 
sequentially by rapid fixations of a positively selected mutation alternating with long periods with no mutation 
(see \cite{orr2005genetic} 
for a review of these models).
However, various empirical evidence \cite{de2006clonal,lang2013pervasive,maddamsetti2015adaptation}
show that in large asexual populations, several mutations can co-occur, which especially can lead to a competition between beneficial mutations. This phenomenon is known as 
clonal interference \cite{gerrish1998fate}, and has a major importance for adaptation of asexual populations, such as bacteria and other prokaryotes, yeasts and other fungi, or cancers.
Consequently in recent years there has been a growing interest in developing experimental studies and theoretical models to analyse clonal interference 
\cite{barton1995linkage,gerrish1998fate,de2006clonal,campos2004modelling,desai2007beneficial,lang2011genetic,lang2013pervasive}.

Most models investigating how clonal interference might affect the probability and time of fixation of beneficial mutations, and thus adaptation, made two important but limiting assumptions: 
first population sizes are constant and independent of the fitness of the individuals, and second fitnesses only depend on the type of the mutations and not on the state of the population, 
i.e. fitnesses are assumed transitive. However, as emphasized by Nowak and Sigmund \cite{nowak2004evolutionary}, there is a reciprocal feedback between adaptation and environmental changes 
because of ecological interactions between individuals, especially competition. Mutation with the highest fitness invades 
the population, but its fitness might depend on the density of the 
other mutations present in the population and might thus change during the course of adaptation. Thus, fitnesses are not necessarily transitive, i.e. selection can be frequency dependent, 
and phenomena such as cyclical dynamics or stable coexistence can occur. Interestingly, non-transitivity of fitnesses has been documented empirically in asexual 
populations \cite{paquin1983relative,kirkup2004antibiotic}.

In this paper we aim at providing and studying a simple individual based stochastic model taking into account the wide variety of competitive interactions
which can be found in nature. 
We show that type dependent competitive interactions are able to generate ecological patterns which are observed but not explained 
by conventional models.
In particular we relax the classical assumption of transitivity between the different mutations 
(see the discussion in Section \ref{discutransi})
and call into questions some classical preconceived ideas about clonal interference (see Propositions \ref{propspeedup}, \ref{propincreaseinvproba}
 and \ref{proannulationeffet}).
To model precisely the interactions between individuals we extend the model introduced in \cite{champagnat2006microscopic} where the author
only considered the occurrence of one mutation. 
The population dynamics, described in Section \ref{model}, is a multitype birth and death Markov process 
with density-dependent competition.
We reflect the carrying capacity of the underlying environment by a scaling parameter $K\in \N$ and state results in the limit for 
large $K$.

Such an eco-evolutionary approach has been introduced by Metz and coauthors \cite{metz1996adaptive} and has been made rigorous in the seminal paper of Fournier 
and M\'el\'eard \cite{fournier2004microscopic}. Then it has been developed by 
Champagnat, M\'el\'eard and coauthors (see
\cite{champagnat2006microscopic,
champagnat2011polymorphic,champagnat2014adaptation} and references therein) for 
the haploid asexual case, by Smadi \cite{smadi2014eco} for the haploid sexual case, and by Collet, M\'el\'eard and Metz \cite{collet2011rigorous} and Coron and coauthors \cite{coron2014stochastic,coron2013slow} for the diploid sexual case.

This work is the first one to study the impact of type dependent competition on clonal interference. 
By using couplings with birth and death processes and comparison with deterministic system, we are able to provide
a complete description of the possible population dynamics 
when two mutations are in competition. We show that a large variety of behaviours can be observed, but that in 
all cases, the dynamics consists in an alternation of long stochastic phases (with a duration 
of order $\log K$) and short phases (with a duration of order $1$) 
which can be approximated by deterministic processes.
In the case of Rock-Paper-Scissors dynamics, we can even have a non bounded number of such alternations (see Proposition \ref{prolizards}).
 The deterministic approximations will be either two-dimensional or three-dimensional Lotka-Volterra competitive systems.
Unlike the two-dimensional case where the final state is completely determined by the signs of the invasion fitnesses 
(defined in \eqref{deffitinv} and \eqref{deffitinv2}),
 the long time behaviour of a three dimensional competitive Lotka-Volterra system can depend on the initial state. 
Moreover, the flows of three dimensional Lotka-Volterra systems
do not necessarily converge to a stable equilibrium but can exhibit cyclical behaviours.
We will see that the time of appearance of the second mutation
is crucial to determine the final population state. For example, the dynamics described in Proposition \ref{prolizards} can only be obtained if the 
second mutation occurs when the fraction of the wild type-individuals in the population is small.
\\

In Sections \ref{model} and \ref{results} we present the model and the main results
of the paper. Sections \ref{firstphase} to \ref{thirdphase} are 
devoted to the study of the dynamics of stochastic and deterministic phases.
Section \ref{proofmainth} is dedicated to the proofs of the main results.
Finally in the Appendices we present technical results (\ref{knownresults}) and provide a complete description of the possible population
dynamics (\ref{completedescr}).

\section{Model}\label{model}

We consider an asexual haploid population and focus on one locus. The individuals may carry three distinct alleles $0$, $1$ or $2$,
and we denote by $\mathcal{E}:=\{0,1,2\}$ the type space.
The population process
\begin{align*}
N^K= (N^{K}(t), t \geq 0)=((N_{0}^{K}(t),N_{1}^{K}(t),N_{2}^{K}(t)), t \geq 0),
\end{align*}
where $N_{i}^{K}(t)$ denotes the number of $i$-individuals 
at time $t$ when the carrying capacity is $K$,
is a multitype birth and death process. The birth rate of $i$-individuals is
\begin{equation}\label{defbirthrate}
 b_i(n)=\beta_i n_i,
\end{equation}
where $\beta_i>0$ is the individual birth rate of $i$-individuals and $n=(n_0,n_1,n_2) \in \Z_+^3$ denotes the current state of the population.
An $i$-individual can die either from a natural death (rate $\delta_i>0$), or from type-dependent competition:
 the parameter $C_{i,j}> 0$ models the impact an 
individual of type $j$ (resp. $i$) has on an individual of type $i$, where $(i,j) \in \mathcal{E}^2$. 
The strength of the competition also depends on
the carrying capacity $K\in \N$, which is a measure of the total quantity of food or space available.
This 
results in the total death rate of individuals carrying the allele $i \in \mathcal{E}$:
\begin{equation}\label{deathrate}
 d^{K}_{i}(n) = \left( \delta_{i} +  \frac{C_{i,0}}{K}n_{0}+  \frac{C_{i,1}}{K}n_{1}+  \frac{C_{i,2}}{K}n_{2} \right) n_{i}.
\end{equation}

Let us now introduce the notions of invasion, final state and time of invasion that we will use in the sequel. 
We say that a type $i \in \mathcal{E}$ invades if there exist $\bar{n} \in \Z_+^3$, $\beta,V>0$ and $K \in \N$ such that
the following event holds:
\begin{equation}\label{definv} I(i,K,\bar{n},\beta,V):=\{ \forall \ \beta \log K < t < e^{VK},\ N^K_i(t)>\bar{n}K\}.  \end{equation}
When $K$ is large, it means that $i$-individuals represent a non negligible fraction of the population during a very long time. 
It is necessary to bound this duration 
(by $e^{VK}$) as a birth and death process with density dependent competition almost surely gets extinct in finite time.
To characterize the long time behaviour of the $0$-, $1$-, and $2$- population sizes, we introduce the following notion of final state set 
for $\bar{n} \in \Z_+^3$, $\eps,\beta,V>0$ and $K \in \N$:
\begin{equation}
\{\bar{n} \in FS(\eps,K,\beta,V)\}:=\Big\{\sup_{\beta\log K < t < e^{VK}}\|N^K(t)/K-\bar{n}\|\leq \eps  \Big\},
\end{equation}
where $\|.\|$ denotes the $L^1$-Norm on $\R^\mathcal{E}$.
Finally, we introduce the time needed for the rescaled population process $N^K/K$ to hit
a vicinity of $\bar{n} \in \Z_+^3$ and stay in the latter during an exponential time.
For $\eps,V>0$ and $K \in \N$:
\begin{equation}\label{definvtime}
 T_{FS(V)}^{(\eps,K)}(\bar{n}):=\inf\{s>0,\forall s < t < e^{VK},\|N^K(t)/K-\bar{n}\|\leq \eps  \}.
\end{equation}

As a quantity summarizing the advantage or disadvantage a mutant with allele type $i$ has in a $j$-population at equilibrium, 
we introduce the so-called invasion fitness $S_{i j}$ through
\begin{equation} \label{deffitinv}
 S_{ij} := \beta_{i} -\delta_{i} - C_{i,j}\bar{n}_{j},
\end{equation}
where the equilibrium density $\bar{n}_{i}$ is defined by
\begin{align*}
\bar{n}_{i}: =\frac{\beta_{i} -\delta_i}{C_{i,i}}.
\end{align*}
The role of the invasion fitness $S_{i j}$ and the 
definition of the equilibrium density $\bar{n}_{i}$ follow from the properties of the two-dimensional competitive Lotka-Volterra 
system:
\begin{equation} \label{S}
\left\{\begin{array}{ll} \dot{n}_i^{(z)}=(\beta_i-\delta_i-C_{i,i}n_i^{(z)}-C_{i,j}n_j^{(z)})n_i^{(z)},& n_i^{(z)}(0)=z_i,\\
        \dot{n}_j^{(z)}=(\beta_j-\delta_j-C_{j,i}n_i^{(z)}-C_{j,j}n_j^{(z)})n_j^{(z)},& n_j^{(z)}(0)=z_j,
       \end{array}
\right.
 \end{equation}
for $z=(z_i,z_j) \in \R_+^2$.
If we assume 
\begin{equation}\label{defnbara}
 \bar{n}_{i}>0,\quad \bar{n}_{j}>0,\quad \text{and} \quad S_{ji}<0<S_{ij},
\end{equation}
then $\bar{n}_{i}$ is the equilibrium size of a monomorphic $i$-population and 
the system \eqref{S} has a unique stable equilibrium $(\bar{n}_i,0)$ and two unstable steady states $(0,\bar{n}_j)$ and $(0,0)$.
Thanks to Theorem 2.1 p. 456 in \cite{ethiermarkov} we can prove that 
if $N_i^K(0)$ and $N_j^K(0)$ are of order $K$, $K$ is large and 
$N_k^K(0)=0$ for $k=\mathcal{E}\setminus \{i,j\}$, 
the process $(N_i^K/K,N_j^K/K)$ is close to the solution of 
\eqref{S} during any finite time interval (see \eqref{eq1lemmeA} for a precise statement).
The invasion 
fitness $S_{ij}$ corresponds to the \textit{per capita} initial growth rate of the mutant $i$ when it appears in a monomorphic 
population of individuals $j$ at their equilibrium size $\lfloor \bar{n}_jK\rfloor$.
Hence the dynamics of the type $i$ is very dependent on the properties of the system \eqref{S}. It is proven in 
\cite{champagnat2006microscopic} 
that under Condition \eqref{defnbara} one mutant $i$ appearing in a monomorphic $j$-population at its equilibrium size 
$\lfloor \bar{n}_j K \rfloor$
has a positive probability to fix.
More precisely, {for $\eps>0$, there exist two finite constants $\beta(\eps)$ and $V(\eps)$ such that
\begin{equation}\label{invichezj}
    \P\Big(  \bar{n}_i e_i \in FS(\eps,K,\beta(\eps),V(\eps)) \Big|
N^K(0)
=e_i+ \lfloor \bar{n}_jK \rfloor e_j\Big)=
\frac{S_{ij}}{\beta_i}+O_K(\eps),
\end{equation}
where $(e_{i}, i\in \mathcal{E} )$ is the canonical 
basis of $\R^\mathcal{E}$ and $O_K(\eps)$ is a function of $K$ and $\eps$ satisfying
\begin{equation}\label{defOKeps}
 \limsup_{K \to \infty}|O_K(\eps)|\leq c \eps,
\end{equation}
for a finite $c$. Similarly, 
\begin{equation}\label{noninvichezj}
 \lim_{K \to \infty} \P\Big(  \bar{n}_j e_j \in FS(\eps,K,\beta(\eps),V(\eps))  \Big|
N^K(0)
=e_i+ \lfloor \bar{n}_jK \rfloor e_j\Big)= 1-
\frac{S_{ij}}{\beta_i}+O_K(\eps).
\end{equation}
Moreover, the invasion time of the mutant $i$, as defined in \eqref{definvtime}, satisfies for a finite $c$
\begin{equation}\label{invtime}\P\Big( (1-c\eps)\frac{\log K}{{S}_{ij}}<T^{(\eps,K)}_{FS(V(\eps))}( \bar{n}_i e_i)<
(1+c\eps)\frac{\log K}{{S}_{ij}}\Big|N^K(0)
=e_i+ \lfloor \bar{n}_jK \rfloor e_j \Big)=\frac{S_{ij}}{\beta_i}+O_K(\eps).\end{equation}}

If we assume on the contrary
\begin{equation}\label{defnbara2}
 \bar{n}_{i}>0,\quad \bar{n}_{j}>0, \quad S_{ij}>0, \quad S_{ji}>0 \quad \text{and}\quad C_{i,i}C_{j,j}\neq C_{i,j}C_{j,i} ,
\end{equation}
then the system \eqref{S} has three unstable steady states $(\bar{n}_i,0)$, $(0,\bar{n}_j)$ and $(0,0)$ and 
a unique stable equilibrium $(\bar{n}_{ij}^{(i)},\bar{n}_{ij}^{(j)})$, where 
\begin{equation} \label{defnij}\bar{n}_{ij}^{(i)}=\frac{C_{j,j}(\beta_i-\delta_i)-C_{i,j}(\beta_j-\delta_j)}
{C_{i,i}C_{j,j}- C_{i,j}C_{j,i}} \quad \text{and} \quad
        \bar{n}_{ij}^{(j)}=\frac{C_{i,i}(\beta_j-\delta_j)-C_{j,i}(\beta_i-\delta_i)}
{C_{i,i}C_{j,j}- C_{i,j}C_{j,i}}.
 \end{equation}
It is proven in 
\cite{champagnat2011polymorphic} 
that under Condition \eqref{defnbara2} one mutant $i$ has a positive probability to invade then coexist with a $j$-population during 
a long time.
More precisely, for  $\eps>0$, there exist two finite constants $\beta(\eps)$ and $V(\eps)$ such that
\begin{equation*}
 \P\Big( \bar{n}_{ij}^{(i)} e_i+\bar{n}_{ij}^{(j)} e_j \in FS(\eps,K,\beta(\eps),V(\eps)) \Big|
N^K(0)
=e_i+ \lfloor \bar{n}_jK \rfloor e_j\Big)=
\frac{S_{ij}}{\beta_i}+O_K(\eps),
\end{equation*}
\begin{equation*}
\P\Big( \bar{n}_j e_j \in FS(\eps,K,\beta(\eps),V(\eps)) \Big|
N^K(0)
=e_i+ \lfloor \bar{n}_jK \rfloor e_j\Big)= 1-
\frac{S_{ij}}{\beta_i}+O_K(\eps).
\end{equation*}
Moreover, the invasion time of the mutant $i$, as defined in \eqref{definvtime}, satisfies for a finite $c$
\begin{equation*}\P\Big((1-c\eps)\frac{\log K}{{S}_{ij}}<T^{(\eps,K)}_{FS(V(\eps))}(\bar{n}_{ij}^{(i)} e_i+\bar{n}_{ij}^{(j)} e_j)<
(1+c\eps)\frac{\log K}{{S}_{ij}}\Big| N^K(0)
=e_i+ \lfloor \bar{n}_jK \rfloor e_j \Big)=\frac{S_{ij}}{\beta_i}+O_K(\eps).\end{equation*}

Analogously, the fate of a mutation of type $k$ occurring when the types $i$ and $j$ coexist in a population 
of large carrying capacity $K$, or the result of the competition between the three types of populations when they have all a 
size of order $K$
is very dependent on the properties of the three-dimensional competitive Lotka-Volterra system:
\begin{equation} \label{S2}
\left\{\begin{array}{ll}
	\dot{n}_i^{(z)}=(\beta_i-\delta_i-C_{i,i}n_i^{(z)}-C_{i,j}n_j^{(z)}-C_{i,k}n_k^{(z)})n_i^{(z)},& n_i^{(z)}(0)=z_i,\\
        \dot{n}_j^{(z)}=(\beta_j-\delta_j-C_{j,i}n_i^{(z)}-C_{j,j}n_j^{(z)}-C_{j,k}n_k^{(z)})n_j^{(z)},& n_j^{(z)}(0)=z_j,\\
	\dot{n}_k^{(z)}=(\beta_k-\delta_k-C_{k,i}n_i^{(z)}-C_{k,j}n_j^{(z)}-C_{k,k}n_k^{(z)})n_k^{(z)},& n_k^{(z)}(0)=z_k,
       \end{array}
\right.
 \end{equation}
where $z= (z_i,z_j,z_k) \in \R_+^3$.
 Similarly as in \eqref{deffitinv} we can define the invasion fitness of the type $k$ in a two-type $i/j$-population,
\begin{equation} \label{deffitinv2}
 S_{kij}= S_{kji} := \beta_{k} -\delta_{k} - C_{k,i}\bar{n}_{ij}^{(i)}- C_{k,j}\bar{n}_{ij}^{(j)},
\end{equation}
where we recall Definition \eqref{defnij}. It corresponds to the initial \textit{per capita} growth rate of an individual of type $k$ appearing in a large population 
in which the individuals of type $i$ and $j$ are at their coexisting equilibrium, 
$(\lfloor \bar{n}_{ij}^{(i)}K \rfloor ,\lfloor \bar{n}_{ij}^{(j)}K \rfloor)$.\\

The class of the three-dimensional competitive Lotka-Volterra sytems has been studied in detail by Zeeman and coauthors 
\cite{zeeman1993hopf,zeeman1998three,zeeman2003local}.
They exhibit much more variety than the two-dimensional ones, and we will recall the long time behaviour of their flows in Section \ref{section3dim}.
The complexity of the three-dimensional competitive Lotka-Volterra systems entails a broad variety of dynamics
in the case of two mutations occurring successively
in a population. In particular, depending on the population state when the second mutation occurs, the new type can be 
advantageous or not when rare: for example, we can have $S_{20}>0$ and $S_{210}<0$, which implies that if the mutant $2$ occurs in 
a $0$-population at equilibrium the $2$-population size has a positive probability to hit a non-negligible fraction of the population, whereas if it 
appears in a two-type $0/1$-population at its coexisting equilibrium it gets extinct immediately.\\

In the sequel, we assume that a mutant of type $1$ occurs in a population of type $0$ (also called wild type) at 
equilibrium at time $0$, and that the mutants $1$ are beneficial when rare. In other words we assume

\begin{ass}\label{asscondinitiale}
$$N^K(0)=\lfloor \bar{n}_0 K\rfloor e_0+e_1 \quad \text{and} \quad  S_{10}>0. $$
\end{ass}

\begin{figure}[h]
\centering
\includegraphics[width=10cm,height=5cm]{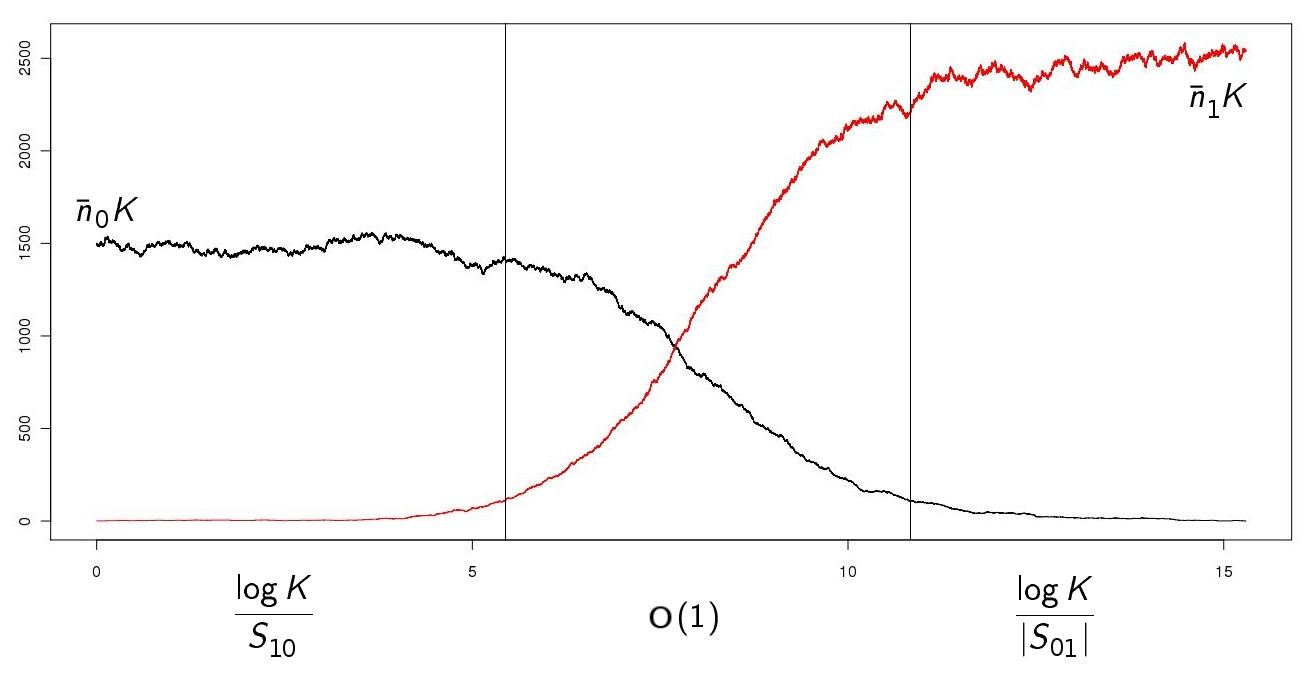}
\caption{The three phases of a mutant invasion. 
The y-axis corresponds to the two type population sizes ($0$ in black, $1$ in red), and the x-axis to the time.
In this simulation, 
$K=1000$, $(\beta_0, \beta_1) =(2,3), \delta_i=0.5, C_{i,j}=1, (i,j) \in \{0,1\}^2$.}
\label{figthreephases}
\end{figure}

Let us describe in few words what happens when there is only the first mutation (we refer to \cite{champagnat2006microscopic} for the proofs).
We can distinguish three phases in the mutant invasion (see Figure \ref{figthreephases}): an 
 initial phase in which the fraction of mutant-individuals does not exceed a fixed value $\eps>0$ and where the dynamics of the wild-type 
 population is nearly undisturbed by the invading type. A second phase where both types account for a non-negligible percentage of the 
 population and where the dynamics of the population can be well approximated by the system \eqref{S}. 
 And finally a third phase where either the roles of the types are interchanged and the wild-type population is near extinction 
 (when \eqref{defnbara} holds), 
or the two types stay close to their coexisting equilibrium (when \eqref{defnbara2} holds).
The duration of the invasion process is of order $\log K $ with a deterministic phase which only 
lasts an amount of time of order 1.
More precisely, during the first phase, when the population $1$ has a size smaller than $\lfloor\eps K \rfloor$, 
this latter can be approximated by a 
supercritical birth and death process, with respective individual birth and death rates 
$$ \beta_1 \quad  \text{ and } \quad \delta_1+\frac{C_{1,0}}{K} \bar{n}_0 K= \beta_1-S_{10}<\beta_1.$$
Hence from well known properties of supercritical birth and death processes (see \cite{MR2047480} for example) we know that 
with a probability close to $S_{10}/\beta_1$ the $1$-population size hits the value $\lfloor \eps K\rfloor$ in a time 
of order $\log K/S_{10}$ for large $K$. The second phase can be approximated 
by the dynamical system \eqref{S} with $(i,j)=(0,1)$ and the system takes a time of order $1$ to get close to 
$(N_0^K,N_1^K)=(\lfloor \bar{n}_{01}^{(0)}K\rfloor,\lfloor \bar{n}_{01}^{(1)}K\rfloor)$ (if $S_{01}>0$) or to 
$(N_0^K,N_1^K)=(\lfloor \eps^2K\rfloor,\lfloor\bar{n}_1K\rfloor)$ (if $S_{01}<0$).
Finally, if $S_{01}>0$ the populations $0$ and $1$ stay close to their coexisting equilibrium 
$(\lfloor \bar{n}_{01}^{(0)}K\rfloor ,\lfloor \bar{n}_{01}^{(1)}K\rfloor)$ 
during a time of order $e^{VK}$ for a positive $V$ and if $S_{01}<0$ the $0$-population size is 
comparable to a subcritical birth and death process, with individual birth  and death rates 
$$\beta_0 \quad \text{ and }\quad\delta_0+\frac{C_{0,1}}{K} \bar{n}_1 K= \beta_0+|S_{01}|>\beta_0.$$
Hence it gets extinct almost surely, in a time close to $\log K/|S_{01}|$.\\

To summarize, for large $K$, the first and third phases have a duration of order at least $\log K$ and the second phase has a duration of order $1$. 
Hence if a second mutation appears during the sweep, this occurs with high probability  
during the first or the third phase. We denote by $\alpha \log K$ the time of occurrence of the second mutation and we distinguish three cases:
\begin{enumerate}
 \item[$\bullet$] Either the mutation $2$ occurs during the first phase of the mutant $1$ invasion and the $1$-population size hits $\lfloor \eps K\rfloor $ before the $2$-population,
which corresponds to the Assumptions \ref{asscondinitiale} and:
\begin{ass}\label{ass1}
 $$S_{20}>0\quad \text{and} \quad 0 \vee \Big(\frac{1}{S_{10}}-\frac{1}{S_{20}}\Big)<\alpha < \frac{1}{S_{10}} . $$
\end{ass}
 \item[$\bullet$] Or it occurs during the first phase but the invasion fitness $S_{20}>0$ is large enough for the $2$-population to hit size $\lfloor \eps K \rfloor$ before the $1$-population, 
which corresponds to the Assumptions \ref{asscondinitiale} and:
\begin{ass}\label{ass2}
 $$S_{20}>0\quad \text{and} \quad 0<\alpha<\mathbf{1}_{S_{20}>S_{10}}\Big(\frac{1}{S_{10}}-\frac{1}{S_{20}}\Big)  . $$
\end{ass}
 \item[$\bullet$] Or it occurs during the third phase, which corresponds to the Assumptions \ref{asscondinitiale} and:
\begin{ass}\label{ass3}
 $$\alpha > \frac{1}{S_{10}} \quad \text{and} \quad \left\{\begin{array}{l}
 S_{01} > 0 \\
\text{or}\\
 S_{01} < 0  \text{ and }  \alpha < 1/S_{10}+1/|S_{01}|.
\end{array}\right.  $$
\end{ass}
\end{enumerate}

In Appendix \ref{completedescr} we give a complete description of the possible population dynamics.
We now highlight the biologically relevant outcomes.

\section{Results}\label{results}

\subsection{Transitive versus non-transitive fitnesses} \label{discutransi}

Let us say a word about classical models of clonal interference 
(see \cite{barton1995linkage,gerrish1998fate,de2006clonal,campos2004modelling,desai2007beneficial} for instance).
The population size is constant (or infinite). The fitness $s_i$ of an $i$-individual corresponds to its exponential growth rate (Malthusian parameter) 
and only depends on its type $i$. 
Suppose that in a population of type $0$ a beneficial mutation $1$ is followed by a second beneficial mutation $2$
 before the fixation of the type $1$ individuals. In population genetics, "beneficial" means that $s_1>0$ and $s_2>0$, which implies that 
the mutant populations have a positive probability to escape genetic drift and constitute a positive fraction of the population.
Then there are two possibilities: 
\begin{enumerate}
 \item Either $s_1>s_2$: then the $1$-population outcompetes the $0$- and $2$-populations and the mutation $1$ becomes fixed.
 \item Or $s_2>s_1$: then the $2$-population outcompetes the $0$- and $1$-populations and the mutation $1$ becomes fixed.
\end{enumerate}
If a third mutation (individuals of type $3$) occurs with fitness satisfying $s_3>s_2$, then not only the type $3$ individuals outcompete the type $2$
 individuals, but by transitivity of the total order $>$ on $\R$, they also outcompete the type $1$ and type $0$ individuals, and so on.
In other words, in population genetics of haploid asexuals, the fitnesses are transitive in the sense that if $1$ outcompetes $0$ and $2$ outcompetes $1$, 
then necessarily $2$ outcompetes $0$.

Such a model is natural when competitive interactions between individuals are simple: in an experiment with only 
one limiting resource, beneficial mutations often correspond to an increase of resource consumption efficiency.
But let us imagine an environment with two resources, $A$ and $B$. A mutant which prefers resource $A$ (resp. $B$)
will be favoured in 
a population of individuals which consume preferentially resource $B$ (resp. $A$). 
In this example there is no transitivity.
Moreover, experiments on cancer and viral cells 
have shown the long term coexistence of several mutant strains \cite{holmes1992convergent},
which cannot be explained by a model with transitive fitnesses.

More generally, it is known \cite{nowak2004evolutionary} that ecological interactions cause non-transitive phenotypic interactions.
Hence to model complex interactions, we need an other definition of "fitnesses", to make appear the dependency on the population state. 
This is achieved by the notion of invasion fitnesses (defined in \eqref{deffitinv} and \eqref{deffitinv2}), which naturally 
follows from the individual based model that we have presented.

The main novelty of our approach is to consider type dependent competitive interactions. Indeed, if 
all the competitive interactions $(C_{i,j}, (i,j) \in \mathcal{E}^2)$ have a same value $C$, then 
for three individual's types $i$, $j$ and $k$, the invasion fitnesses satisfy:
 $$ S_{ij}>0 \Longleftrightarrow   S_{ji} = \beta_j-\delta_j- C \bar{n}_i=\beta_j-\delta_j- C (\beta_i-\delta_i)/C=-S_{ij}<0, $$
and 
$$ S_{ij}>0 \text{ and } S_{jk}>0 \Longrightarrow  S_{ik}=S_{ij}+S_{jk}>0 . $$
In other words, it boils down to the case of transitive fitnesses.
We can have a more precise result on the form of competitions allowing non transitive relations between several mutants.
To state this latter we introduce the following order for $ (i,j) \in \mathcal{E}^2$:
$$ i \prec j \Longleftrightarrow S_{ij}<0<S_{ji}, $$
$$ i = j \Longleftrightarrow  S_{ij}.S_{ji}\geq 0  ,$$
and the notation 
\begin{equation} \label{tildeC} \tilde{C}_{i,j}= \frac{C_{i,j}}{C_{j,j}}. \end{equation}
Then we can state the following Lemma, which will be proven in Appendix \ref{knownresults}.

\begin{lem}\label{transivsnontransi}
 Let $C_1\leq C_2$ be two positive real numbers and $i,j,k \in \mathcal{E}$.
 \begin{enumerate}
  \item If
 \begin{equation}\label{transicas1} \Big(C_1 \vee \frac{1}{C_2} \Big)^2>C_2 \quad \text{and} \quad \Big( \frac{1}{C_1} \wedge C_2 \Big)^2<C_1 \end{equation}
then if for every $m\neq n \in \{ i,j,k \}$, $C_1 \leq \tilde{C}_{m,n} \leq C_2$, 
$$ i \prec j \quad \text{and} \quad j \prec k \Longrightarrow i \prec k . $$
\item If
 \begin{equation}\label{transicas2} \Big(C_1 \vee \frac{1}{C_2} \Big)^2>C_2 \quad \text{or} \quad \Big( \frac{1}{C_1} \wedge C_2 \Big)^2<C_1 \end{equation}
then if for every $m\neq n \in \{i,j,k\}$, $C_1 \leq \tilde{C}_{m,n} \leq C_2$, 
$$ i \prec j \quad \text{and} \quad j \prec k \Longrightarrow i \preceq k . $$
\item If
 \begin{equation}\label{transicas3}  \Big(C_1 \vee \frac{1}{C_2} \Big)^2<C_2 \quad \text{and} \quad \Big( \frac{1}{C_1} \wedge C_2 \Big)^2>C_1 \end{equation}
then there exist some ecological parameters $(\rho_m, \tilde{C}_{m,n}, m\neq n \in \{i,j,k\})$ 
such that for every $m\neq n \in \{i,j,k\}$, $C_1 \leq \tilde{C}_{m,n} \leq C_2$ and
$$ i \prec j, \quad j \prec k \quad \text{and} \quad k \prec i  .$$
  \end{enumerate}
\end{lem}

\begin{rem}
 By looking at all the possible subcases we can show the following equivalencies:
 $$ \eqref{transicas1} \Longleftrightarrow C_2^2<C_1 <C_2<1 \quad \text{or} \quad 1<C_1<C_2<C_1^2,  $$
 $$ \eqref{transicas2} \Longleftrightarrow C_1<C_2^2<C_2<1 \quad \text{or} \quad 1<C_1<C_1^2 <C_2,  $$
 $$ \eqref{transicas3} \Longleftrightarrow C_1<1<C_2 . $$
 This shows that transitive relations are more likely when the rescaled competitions $(\tilde{C}_{i,j}, (i,j) \in \mathcal{E}^2)$ are 
 close to each other and both smaller of greater than $1$.
\end{rem}

By allowing dependency of competitive interactions on the individual's types, we are able to model new ecological patterns found in nature 
and describe them mathematically. 
In particular, we will show that it allows us to model coexistence of several 
mutant populations, and complex dynamics as Rock-Paper-Scissors cycles.\\

To present our main results in a simple way, we introduce the following notation,
\begin{equation*}
 \P^{(k)}(.)=\P(.|\text{the mutation(s) $k$ is(are) present}), \quad k \in \{\{1\},\{2\},\{1,2\}\},
\end{equation*}
where we recall that the mutation $1$ occurs at time $0$ and the mutation $2$ at time $\alpha \log K$.

\subsection{Does clonal interference speed up or slow down invasion?} \label{restime}

In population genetic models of clonal interference, the authors concluded that the presence of other mutants slowed down the invasion of the fittest mutant 
as it had to outcompete some individuals fitter than the wild type individuals \cite{muller1932some,muller1964relation,gerrish1998fate,de2006clonal}.
In our individual based model, the result of the competition between two 
mutant subpopulations depends on the state of the total population, and the interplay between different mutations can take various forms.
For some values of the parameters,
clonal interference does indeed slow down the invasion of a beneficial mutant. 
Recall Equations \eqref{invichezj} to \eqref{invtime} about the 
invasion of 
one mutant. Then

\begin{pro} \label{propslowdown}
Let Assumptions \ref{asscondinitiale} and \ref{ass1} be satisfied and suppose:
\begin{equation}\label{cas1slow}
  S_{01},S_{02}<0,\quad S_{12}<0<S_{21}, 
 \quad \text{and} \quad S_{21}<S_{20},
\end{equation}
or 
\begin{equation}\label{cas2slow} S_{01}>0,S_{02}<0,\quad S_{12}<0<S_{201}\quad \text{and} \quad S_{201}<S_{20}. \end{equation}
Then the presence of the mutation $1$does not modify the invasion probability of the mutation $2$ and the final state set of the population: 
for $\eps>0$, there exist two finite constants $\beta(\eps)$ and $V(\eps)$ such that
$$\P^{(k)}(\bar{n}_{2}e_2 \in FS(\eps,K,\beta(\eps),V(\eps)))=\frac{S_{20}}{\beta_2}+O_\eps(K),
\quad \text{for }k = \{1,2\} \text{ or } \{2\}.$$
However the presence of the first mutant slows down the invasion of the second one: there exists 
a positive constant $c$ such that 
for every $\eps>0$:
$$\P^{(1,2)}\Big((1-c\eps)\frac{\log K}{\tilde{S}_{20}}<T^{(\eps,K)}_{FS(V(\eps))}(\bar{n}_{2}e_2)-\alpha \log K<(1+c\eps)\frac{\log K}{\tilde{S}_{20}}\Big)=
\frac{S_{20}}{\beta_2}\frac{S_{10}}{\beta_1}+O_K(\eps),$$
where $\tilde{S}_{20}>{S}_{20}$ is defined by
$$ \frac{1}{\tilde{S}_{20}}=\frac{1}{{S}}+\Big(\frac{1}{S_{10}}-\alpha \Big)\Big(1-\frac{S_{20}}{{S}}\Big) > \frac{1}{S_{20}}, $$
where ${S}= S_{21} $ in case \eqref{cas1slow} and $S_{201}$
in case \eqref{cas2slow}, and
$$\P^{(1,2)}\Big((1-c\eps)\frac{\log K}{{S}_{20}}<T^{(\eps,K)}_{FS(V(\eps))}(\bar{n}_{2}e_2)<
(1+c\eps)\frac{\log K}{{S}_{20}}\Big)=\frac{S_{20}}{\beta_2}\Big(1-\frac{S_{10}}{\beta_1}\Big)+O_K(\eps).$$
\end{pro}

But due to the non-transitive phenotypic interactions that our model is able to take into account, 
the presence of the first mutant can also speed up the invasion of the second one. More precisely we have the following result, 
which is illustrated in Figure \ref{graphprop22}:

\begin{pro}\label{propspeedup}
Let Assumptions \ref{asscondinitiale} and \ref{ass1} be satisfied and suppose:
\begin{equation}\label{cas1speed}
  S_{01},S_{02}<0,\quad S_{12}<0<S_{21}, 
 \quad \text{and} \quad
S_{21}>S_{20},
\end{equation}
or 
\begin{equation}\label{cas2speed} S_{01}>0,S_{02}<0,\quad S_{12}<0<S_{201}\quad \text{and} \quad S_{201}>S_{20}. \end{equation}
Then the presence of the mutation $1$does not modify the invasion probability of the mutation $2$ and the final state of the population: 
{for $\eps>0$, there exist two finite constants $\beta(\eps)$ and $V(\eps)$ such that
$$\P^{(k)}(\bar{n}_{2}e_2 \in FS(\eps,K,\beta(\eps),V(\eps)))=\frac{S_{20}}{\beta_2}+O_\eps(K), \quad \text{for }k = \{1,2\} \text{ or } \{2\}.$$}
However the presence of the first mutant speeds up the invasion of the second one: there exists a positive constant $c$ such that 
for every $\eps>0$:
$$\P^{(1,2)}\Big( (1-c\eps)\frac{\log K}{\tilde{S}_{20}}<T^{(\eps,K)}_{FS(V(\eps))}(\bar{n}_{2}e_2)<(1+c\eps)\frac{\log K}{\tilde{S}_{20}}\Big)=
\frac{S_{20}}{\beta_2}\frac{S_{10}}{\beta_1}+O_K(\eps),$$
where $\tilde{S}_{20}>{S}_{20}$ is defined by
$$ \frac{1}{\tilde{S}_{20}}=\frac{1}{S}+\Big(\frac{1}{S_{10}}-\alpha \Big)
\Big(1-\frac{S_{20}}{S}\Big) < \frac{1}{S_{20}}, $$
where $S= S_{21}$ in case \eqref{cas1speed} and $S= S_{201}$ in case \eqref{cas2speed}, and
$$\P^{(1,2)}\Big( (1-c\eps)\frac{\log K}{{S}_{20}}<T^{(\eps,K)}_{FS(V(\eps))}(\bar{n}_{2}e_2)<(1+c\eps)\frac{\log K}{{S}_{20}}\Big)=
\frac{S_{20}}{\beta_2}\Big(1-\frac{S_{10}}{\beta_1}\Big)+O_K(\eps).$$
\end{pro}

\begin{figure}[h]
\centering
\includegraphics[width=9cm,height=4.5cm]{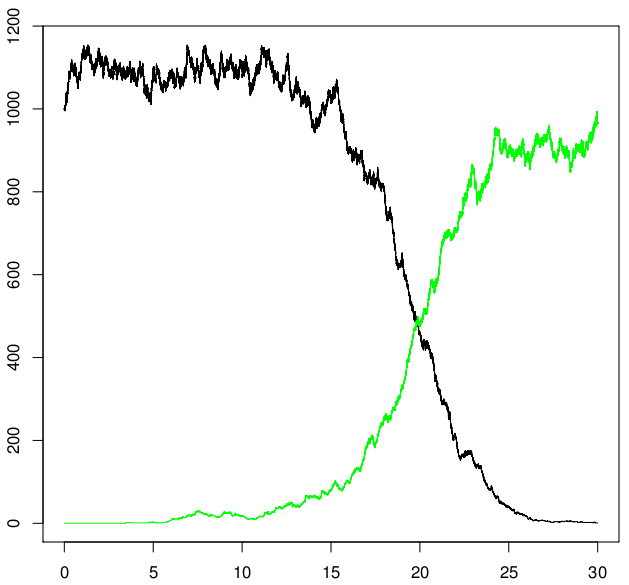}\hspace{.2cm}
\includegraphics[width=6cm,height=4.5cm]{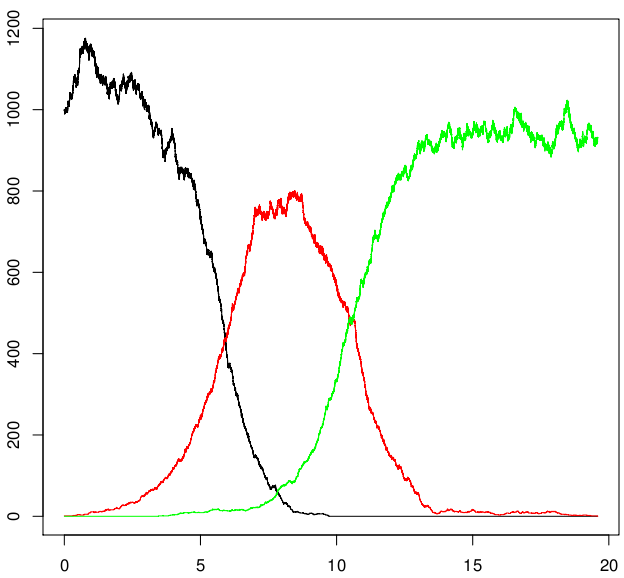}
\caption{Invasion times of the mutation $2$ without and with the first mutation. 
The y-axis corresponds to the three type population sizes ($0$ in black, $1$ in red and $2$ in green), and the x-axis to the time. 
In these simulations, 
$K=1000$, $(\beta_i, \delta_i) =(2,0), i \in \mathcal{E}$, $C_{0,0}=1.8, C_{0,1}=4, C_{0,2}=3, C_{1,0}=1, C_{1,1}=2.3, C_{1,2}= 3, 
C_{2,0}=1.5, C_{2,1}=1, C_{2,2}=2.1$, and $\alpha=0.5$.}
\label{graphprop22}
\end{figure}

\subsection{How does clonal interference modify the invasion probability?}\label{resfix}

Previous models predicted that clonal interference could only decrease the invasion probability of a mutant \cite{de2006clonal}.
It was a direct consequence of the fact that the fate of a competition between two mutants was only dependent on the relative 
values of their fitnesses.
Moreover, this implied that individuals with different phenotypes could not coexist for a long time.
In our model, both decrease and increase of the invasion probability due to clonal interference may occur, depending on the 
competitive interactions between individuals and the time of appearance of the second mutation. Moreover, long term coexistence of several beneficial mutations 
are allowed when we do not assume that the mutant fitnesses are totally ordered. Recall the definition of invasion in \eqref{definv}. Then we have the following result:

\begin{pro}\label{propdecreaseinvproba}
Let Assumption \ref{asscondinitiale} and one of the following conditions hold:
$$\left\{ \begin{array}{l}
    S_{01}>0, \quad S_{201}<0 \quad  \text{and Assumption \ref{ass1} holds}, \\
S_{01}<0, \quad S_{21}<0  \quad \text{and Assumption \ref{ass1} holds}\\
S_{01}<0, \quad S_{21}<0,  \quad S_{20}>0 \quad \text{and Assumption \ref{ass3} holds}
   \end{array}\right.
 $$ 
Then for every $n \in \R_+^3 \setminus \{0\} $ and $\beta,V>0$:
$$\lim_{K \to \infty}\P^{(1,2)}(I(2,K,n,\beta,V))=0,$$
whereas
there exist $n \in \R_+^3 \setminus \{0\}$ and $\beta,V>0$ such that:
$$
\lim_{K \to \infty}\P^{(2)}(I(2,K,n,\beta,V))=\frac{S_{20}}{\beta_2}.$$
\end{pro}

Notice that even if we recover here a classical result, saying that a mutant can get extinct because of the competition with an other mutant, 
we do not require $S_{10}>S_{20}>0$, which would be the equivalent of assumptions done in population genetic models. 

\begin{pro}\label{propincreaseinvproba}
Let Assumptions \ref{asscondinitiale} and \ref{ass3} hold, and suppose that $S_{20}<0$. Then under one of the following additional conditions
\begin{equation}\label{cond1proincrease2} S_{01}<0, \ S_{21}>0 \ \text{and} \ \left\{ \begin{array}{l}
    0<\alpha-\frac{1}{S_{10}} < \frac{1}{|S_{01}|}-\frac{1}{S_{21}} \text{ and }
\left\{\begin{array}{l}
 S_{12}>0 \\
\text{or}\\
S_{12}<0 \ \text{ and } \ S_{02}<0
\end{array} \right.\\
\text{or}\\
  \alpha-\frac{1}{S_{10}} > \frac{1}{|S_{01}|}-\frac{1}{S_{21}},
    \end{array}\right.,
 \end{equation}
 or
\begin{equation}\label{cond1proincrease} S_{01}>0 \quad \text{and}\quad S_{201}>0, \end{equation}
there exist $n \in \R_+^3 \setminus \{0\}$ and $\beta,V>0$ such that:
$$\lim_{K \to \infty}\P^{(1,2)}(I(2,K,n,\beta,V))=\frac{S}{\beta_2}\frac{S_{10}}{\beta_1},$$
where $S=S_{21}$ in case \eqref{cond1proincrease2} and $S= S_{201}$ in case \eqref{cond1proincrease},
whereas 
for every $n \in \R_+^3 \setminus \{0\}$ and $\beta,V>0$:
$$\lim_{K \to \infty}\P^{(2)}(I(2,K,n,\beta,V))=0.$$
\end{pro}

Concerning the interplay of invasion and clonal interference, let us mention recent works
which have taken into account the case where 
many beneficial mutations occur before any can fix
\cite{kim2005adaptation,bollback2007clonal,desai2007beneficial,lang2011genetic,lang2013pervasive}.
The authors still assume transitivity of mutant fitnesses, but consider a regime of frequent mutations 
(high mutation rate or very large population).
New mutations constantly occur in individuals already carrying other mutations, in their way of invasion, and the fate of a mutation depends on the 
genetic background of the individual where it occurs more than on its intrinsic advantage.
They argue that this dynamical equilibrium is a way to preserve genetic diversity despite clonal interference, where  the amount of variation results 
in a subtle balance between selection, which reduces it, and new mutations, which increase it.
This approach is interesting and relates on experimental data which confirm that populations with so frequent mutations do exist
\cite{lang2013pervasive}.
However, due to the transitivity assumption, 
the authors need to assume that a large number of mutants co-occur in order to explain the possible coexistence of several types in the population, which is not necessary in our model.\\

\subsection{When beneficial mutations annihilate adaptation?}

An other interesting phenomenon can happen in our model: the occurrence of the second mutation can annihilate the effects of the first one and 
lead to the final fixation of the wild type population $0$ which would have been outcompeted by the first mutation alone.
Such a phenomenon is also impossible in case of transitive fitnesses, 
as in this setting $S_{10}>0$ and $S_{21}>0$ necessarily imply $S_{20}>0$.
Proposition \ref{proannulationeffet} is illustrated in Figure \ref{illuprop5}.

\begin{pro}\label{proannulationeffet}
Let Assumptions \ref{asscondinitiale} and \ref{ass3} hold and suppose that 
$$ S_{01}<0,\quad S_{12}<0<S_{21},\quad S_{20}<0< S_{02}\quad 
\text{and} \quad  \frac{S_{02}}{|S_{12}||S_{01}|}<\alpha -\frac{1}{S_{10}}+\frac{1}{S_{21}}<\frac{1}{|S_{01}|} . $$
Then for every $\eps>0$, there exist two finite constants $\beta(\eps)$ and $V(\eps)$ such that
$$\P^{(1,2)}(\bar{n}_{0}e_0 \in FS(\eps,K,\beta(\eps),V(\eps)))=\frac{S_{10}}{\beta_1}\frac{S_{21}}{\beta_2}+O_K(\eps) .$$
\end{pro}

\begin{figure}[h]
\centering
\includegraphics[width=10cm,height=5cm]{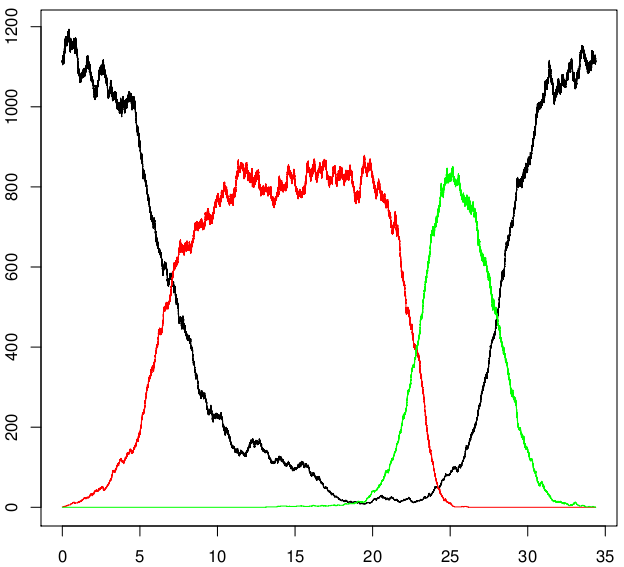}
\caption{Illustration of Proposition \ref{proannulationeffet}. 
The y-axis corresponds to the three type populations sizes ($0$ in black, $1$ in red and $2$ in green), and the x-axis to the time. 
In these simulations, 
$K=1000$, $(\beta_i, \delta_i) =(2,0), i \in \{0,1,2\}$, $C_{0,0}=1.8,  
C_{1,0}=C_{2,1}=1, C_{1,1}=2.3, C_{1,2}=2 C_{0,1}= 5, 
C_{2,0}=2 C_{0,2}=3, C_{2,2}=2.1$, and $\alpha=1.9$.}
\label{illuprop5}
\end{figure}

\subsection{At the origin of the Rock-Paper-Scissors cycles}\label{resRPS}
Rock-Paper-Scissors (RPS) is a children's game where rock beats scissors, which beat paper, which in turn beats 
rock. Such competitive interactions between morphs or species in nature can lead to cyclical dynamics, and 
have been documented in various ecological systems 
\cite{buss1979competitive,taylor1990complex,sinervo1996rock,kerr2002local,kirkup2004antibiotic,cameron2009parasite,nahum2011evolution}.
Let us describe two examples of such cycles. The first one \cite{sinervo1996rock} is concerned with pattern of 
sexual selection on some male lizards. 
Males is associated to their throat colours, which have three morphs.
Type $0$ individuals (orange throat) are monogamous and very aggressive.
They control a large territory.
Type $1$ individuals (dark-blue throat) are
polygamous and less efficient in defending their territory, which is smaller, having to split their efforts on several females. 
Finally type $2$ individuals (prominent yellow stripes on the throat, similar to 
receptive females) do not engage in female-guarding behavior 
but roam around in search of sneaky matings. As a consequence of these different strategies, the type $0$ outcompetes the type $1$, which outcompetes 
the type $2$, which in turn outcompetes the type $0$.
The second example \cite{kirkup2004antibiotic} is concerned with the interactions 
between three strains of \textit{Escherichia coli} bacteria. Type $0$ individuals release toxic
colicin and produce an immunity protein. Type $1$ individuals produce the immunity protein only. Type $2$ individuals produce neither 
toxin nor immunity. Then type $0$ is defeated by type $1$ (because of the cost of toxic colicin production), which is defeated by type $2$, 
(because of the cost of immunity protein production), which in turn is defeated by type $0$ (not protected against toxic colicin).

Neumann and Schuster \cite{neumann2007continuous} modeled such interactions by a three dimensional 
competitive Lotka-Volterra 
system. In particular they proved that migrations or recurrent mutations were not necessary ingredients to obtain limit cycles, as it was 
assumed in previous models (see \cite{czaran2002chemical,schreiber2013spatial} for example).
They studied the long time behaviour of the system but payed little attention to initial conditions, only assuming ``the presence of all three strains in one 
homogeneous medium''. But the question of initial conditions is crucial. Indeed, how to explain the appearance of such cycles 
whereas when only two strains are present one of them is outcompeted by the other one and disappears?
Our simple model provides a framework explaining how a cyclical RPS dynamics emerges in an ecological system thanks to the 
interplay of two successive mutations. Proposition \ref{prolizards} is illustrated in Figure \ref{RPScycles}.

\begin{pro}\label{prolizards}
Let Assumptions \ref{asscondinitiale} and \ref{ass3} and the following inequalities hold: 
\begin{equation}\label{cond1RPS}
 S_{01}<0<S_{10}, \quad S_{12}<0<S_{21}, \quad S_{20}<0<S_{02},
\end{equation}
and
\begin{equation}\label{cond2RPS}
 0< \alpha -\frac{1}{S_{10}}<\min \Big(\frac{1}{|S_{01}|}, \frac{S_{02}}{|S_{12}||S_{01}|},\frac{S_{02}S_{10}}{|S_{12}||S_{01}||S_{20}|}  \Big)-\frac{1}{S_{21}}.
\end{equation}  
Then for every $l \in \N$, if we call ``cycle'' the interval between two local (non null) maxima of the type $1$ 
population between which type $2$ and 
type $0$ populations also hit one local (non null) maximum, and $\mathcal{N}$ the number of cycles
$$\lim_{K \to \infty}\P^{(1,2)}(\mathcal{N}\geq l)=\frac{S_{10}}{\beta_1}\frac{S_{21}}{\beta_2}.$$
Moreover, if we denote by $D_l$ the duration of the $k$th cycle and introduce: 
$$ T(\alpha,l,K):=\Big(\alpha -\frac{1}{S_{10}}+\frac{1}{S_{21}} \Big) \Big(1+\frac{|S_{01}|}{S_{02}} +\frac{|S_{01}||S_{12}|}{S_{02}S_{10}} \Big)
\Big(\frac{|S_{01}||S_{12}||S_{20}|}{S_{02}S_{21}S_{10}} \Big)^{l-1}\log K,$$
then $D_k$ satifies for a finite $c$
$$\P^{(1,2)}\Big(\{\mathcal{N}\geq l\}, (1-c\eps)T(\alpha,l,K)<D_k<(1+c\eps)T(\alpha,l,K)\Big)=
\frac{S_{10}}{\beta_1}\frac{S_{21}}{\beta_2}+O_K(\eps).$$
\end{pro}

\begin{figure}[h]
\centering
\includegraphics[width=10cm,height=5cm]{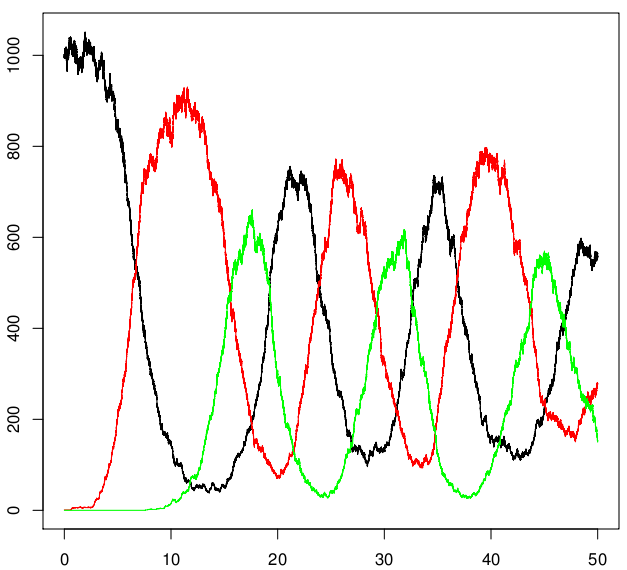}
\caption{RPS cycles. 
The y-axis corresponds to the three type population sizes ($0$ in black, $1$ in red and $2$ in green), and the x-axis to the time. 
In this simulation, 
$K=1000$, $(\beta_i, \delta_i) =(2,0), i \in \{0,1,2\}$, $C_{0,0}=C_{1,1}=C_{2,2}=2, C_{0,1}=2.5, C_{0,2}=C_{1,0}=C_{2,1}=1, 
C_{1,2}= C_{2,0}=3$, and $\alpha=1.1$.}
\label{RPScycles}
\end{figure} 

\begin{rem}\label{remarkRPS}
 We say that System \eqref{S2} is permanent if there exists a compact attractor $\mathcal{B} \subset int(\R_+^3) $ of its solutions, whose 
basin of attraction is $int(\R_+^3) $. Under Condition \eqref{cond1RPS},
Theorem 2 in \cite{neumann2007continuous} states that \eqref{S2} is permanent if and only if 
$$ |S_{01}||S_{12}||S_{20}|<S_{02}S_{21}S_{10}.$$
This condition is satisfied under the assumptions of Proposition \ref{prolizards}.
Hence if one of the types gets extinct, this is due to the demographic stochasticity and not to the behaviour of the approximating dynamical system.

Moreover, we are not able to know in general if the interior fixed point is globally attracting of if there exist stable periodic orbits for the flows 
of the three dimensional deterministic Lotka-Volterra system. In Appendix \ref{knownresults} we give two examples of systems satisfying 
the conditions of Proposition \ref{prolizards} but with distinct long time behaviours.
\end{rem}

\section{First phase}\label{firstphase}

In Sections \ref{firstphase} to \ref{thirdphase} we describe the dynamics of the successive phases.
For sake of readability, we do not indicate anymore the superscript $^{(k)}$ of the probabilities. 
The context will make clear the mutations which occur.
The first phase is rigorously defined as follows:
\begin{equation*}
 \text{Phase $1$}:=\{t\geq 0, N_1^K(t)+N_2^K(t)\geq 1, \sup_{i \in \{1,2\}}N_i^K(t)< \lfloor \eps K \rfloor\}.
\end{equation*}

Under Assumption \ref{ass3} the second mutation occurs during the third phase. This corresponds to the case already studied in 
\cite{champagnat2006microscopic} and we will recall the outcomes in this case in Section \ref{phase1Ass3}. Under Assumption \ref{ass1} or \ref{ass2} the second mutation 
occurs during the first phase and we have to study the resulting dynamics during the first phase.

\subsection{Assumption \ref{ass1}}
Let $i$, $j$ and $k$ be distinct in $\{0,1,2\}$, $\eps>0$ and $K \in \N$.
We introduce a finite subset of $\N$ containing 
the equilibrium size of a monomorphic $i$-population,
\begin{equation} \label{compact}I_\eps^{(K,i)}:= \Big[K\Big(\bar{n}_i-2\eps \frac{C_{i,j}+C_{i,k}}{C_{i,i}}\Big),K\Big(\bar{n}_i+2\eps \frac{C_{i,j}+C_{i,k}}{C_{i,i}}\Big)\Big] \cap \N , \end{equation}
and the stopping times $T^{(K,i)}_a$, $T^{(K,ij)}_a$ and $\tilde{T}^{(K,i)}_\eps$, which denote respectively the hitting time of size
$\lfloor a \rfloor$ for $a \in \R_+$ by 
the population of type $i$ and by the total population of types $i$ and $j$, and the exit time of $I_\eps^{(K,i)}$ by the population of type $i$,
\begin{equation} \label{TKTKeps} T^{(K,i)}_a := \inf \Big\{ t \geq 0, N^K_i(t)= \lfloor a \rfloor \Big\}, \end{equation}
 \begin{equation}
 T^{(K,ij)}_a := \inf \Big\{ t \geq 0, N_i^K(t)+N^K_j(t)= \lfloor a \rfloor \Big\}, \end{equation}
\begin{equation} \label{TKTKeps1} \tilde{T}^{(K,i)}_\eps := \inf \Big\{ t \geq 0, N^K_i(t)\notin I_\eps^{(K,i)} \Big\}.  \end{equation}

Finally, we introduce a finite subset of $\N$ which may contain the type $2$ population size at the end of the first phase.
\begin{equation}\label{defJepsK}
 J_\eps^{(K,2)}:=\Big[K^{S_{20}(\frac{1}{S_{10}}-\alpha-\eps)},
K^{S_{20}(\frac{1}{S_{10}}-\alpha+\eps)}\Big]\cap \N,
\end{equation}
where we recall that $\alpha \log K$ is the time of occurrence of the second mutation, 
and the invasion fitnesses have been defined in \eqref{deffitinv}.
Then we have the following possible states with positive probability at the end of the first phase.

\begin{lem}\label{lemphase11}
Under Assumptions \ref{asscondinitiale} and \ref{ass1}, there exists a positive constant $M_1$ such that
\begin{eqnarray*}
& (a)&  \P(T^{(K,i)}_0<T^{(K,i)}_{\eps K}<\tilde{T}^{(K,0)}_\eps, \forall i \in \{1,2\})=
\Big(1-\frac{S_{10}}{\beta_1}\Big)\Big(1-\frac{S_{20}}{\beta_2}\Big)+O_K(\eps), \nonumber \\
&(b) &  \P(T^{(K,i)}_{\eps K}<T^{(K,i)}_0<\tilde{T}^{(K,0)}_\eps,T^{(K,j)}_0<T^{(K,j)}_{\eps K}<\tilde{T}^{(K,0)}_\eps)=
\frac{S_{i0}}{\beta_i}\Big(1-\frac{S_{j0}}{\beta_j}\Big)+O_K (\eps) , i \neq j \in \{1,2\}, \nonumber \\
&(c) &  \P\Big(T^{(K,1)}_{\eps K}<\tilde{T}^{(K,0)}_\eps,N_2^K(T^{(K,1)}_{\eps K}) \in J_{M_1\eps}^{(K,2)}\Big)=
\frac{S_{10}}{\beta_1}\frac{S_{20}}{\beta_2}+O_K(\eps).
\end{eqnarray*}
\end{lem}

\begin{proof}
In the vein of Fournier and M\'el\'eard \cite{fournier2004microscopic} we represent the population process in terms 
of Poisson measures. Let $Q_0(ds,d\theta)$, $Q_1(ds,d\theta)$ and $Q_2(ds,d\theta)$ be three independent 
Poisson random measures on $\R_+^2$ with intensity $ds d\theta$, and recall that $(e_{i}, i\in \{0,1,2\} )$ is the canonical 
basis of $\R^3$.
Then the process $N^K$ can be written as follows:
\begin{multline}\label{defN}
N^K(t)= N^K(0)+
\underset{i=0}{\overset{1}{\sum}}\int_0^t\int_{\R_+} e_i \Big[ \mathbf{1}_{\theta \leq \beta_iN_i^K(s^-)}-
\mathbf{1}_{0<\theta-\beta_iN_i^K(s^-)\leq d_i^K(N^K({s^-}))}\Big] Q_i(ds,d\theta) \\
  +  e_2\mathbf{1}_{t\geq  \alpha \log K} \Big(1+ \int_{\alpha \log K}^t\Big[ \mathbf{1}_{\theta \leq \beta_2N_2^K(s^-)}-
\mathbf{1}_{0<\theta-\beta_2N_2^K(s^-)\leq d_2^K(N^K({s^-}))}\Big] Q_2(ds,d\theta) \Big) ,
\end{multline}
where $(d_i^K, i \in \{0,1,2\})$ have been defined in \eqref{deathrate}.
The idea is to couple the population process with birth and death processes to get bounds on the different hitting times.
To this aim, let us introduce approximations of the so called rescaled invasion fitnesses ${S_{ij}}/{\beta_i}$,
for $i \in \{1,2\}$ and $\eps>0$ small enough:
\begin{equation}\label{defsi0-}
 s_{i0}^{(\eps,-)}:= \frac{1}{\beta_i}\Big(\beta_i-\delta_i -C_{i,0}\Big( \bar{n}_0+2\eps \frac{C_{0,1}+C_{0,2}}{C_{0,0}} \Big)-(C_{i,1}+C_{i,2})\eps\Big)
\end{equation}
\begin{equation}\label{defsi0+}
 s_{i0}^{(\eps,+)}:=\frac{1}{\beta_i}\Big( \beta_i-\delta_i -C_{i,0}\Big( \bar{n}_0-2\eps \frac{C_{0,1}+C_{0,2}}{C_{0,0}} \Big)\Big).
\end{equation}
These real numbers satisfy for $\eps$ small enough
\begin{equation}\label{diffs+s-} 0<s_{i0}^{(\eps,-)}\leq \frac{S_{i0}}{\beta_i}\leq s_{i0}^{(\eps,+)}<1, \quad \text{and} \quad \Big|s_{i0}^{(\eps,+)}-s_{i0}^{(\eps,+)}\Big| \leq
\frac{1}{\beta_i} \Big(4\frac{C_{0,1}+C_{0,2}}{C_{0,0}}+C_{i,1}+C_{i,2} \Big) \eps .\end{equation}
Thanks to these definitions we can introduce, for $* \in \{-,+\}$ the processes 
\begin{multline*}
N_0^{(\eps,*)}(t)=\lfloor\bar{n}_0 K\rfloor+
\int_0^t\int_{\R_+}  \Big[ \mathbf{1}_{\theta \leq \beta_0N_0^{(\eps,*)}(s^-)}-\\
\mathbf{1}_{0<\theta-\beta_0N_0^{(\eps,*)}(s^-)\leq (\delta_0+C_{0,0}N_0^{(\eps,*)}(s^-)+\mathbf{1}_{\{*=-\}}\eps(C_{0,1}+C_{0,2}))
N_0^{(\eps,*)}(s^-)}\Big] Q_0(ds,d\theta),
\end{multline*}
and the supercritical birth and death processes
\begin{equation}\label{defN1star}
N_1^{(\eps,*)}(t)=1+
\int_0^t\int_{\R_+}  \Big[ \mathbf{1}_{\theta \leq \beta_1N_1^{(\eps,*)}(s^-)}-
\mathbf{1}_{0<\theta-\beta_1N_1^{(\eps,*)}(s^-)\leq (1-s_{10}^{(\eps,*)})\beta_1N_1^{(\eps,*)}(s^-)}\Big] Q_1(ds,d\theta),
\end{equation}
and 
\begin{equation}\label{defN2star}
N_2^{(\eps,*)}(t)= \mathbf{1}_{t\geq  \alpha \log K}\Big(1+
\int_{\alpha \log K}^t\int_{\R_+}  \Big[ \mathbf{1}_{\theta \leq \beta_2N_2^{(\eps,*)}(s^-)}-
\mathbf{1}_{0<\theta-\beta_2N_2^{(\eps,*)}(s^-)\leq (1-s_{20}^{(\eps,*)})\beta_2N_2^{(\eps,*)}(s^-)}\Big] Q_2(ds,d\theta)\Big).
\end{equation}
Then recalling Definition \eqref{TKTKeps} we get
\begin{equation}\label{couplage0} N_0^{(\eps,-)} (t)\leq N_0^{K}(t)\leq N_0^{(\eps,+)}(t) \quad \text{a.s.}\quad \forall \ t < 
T_{\eps K}^{(K,1)}\wedge T_{\eps K}^{(K,2)}  
\end{equation}
and for $i \in \{1,2\}$,
\begin{equation}\label{couplage12} N_i^{(\eps,-)}(t)\leq N_i^{K}(t)\leq N_i^{(\eps,+)}(t) \quad \text{a.s.}\quad \forall \ t < 
T_{\eps K}^{(K,1)}\wedge T_{\eps K}^{(K,2)}   \wedge \tilde{T}^{(K,0)}_\eps. \end{equation}
Moreover by construction for every $(*_1,*_2) \in \{-,+\}^2$ the processes $N_1^{(\eps,*_1)}$ and $N_2^{(\eps,*_2)}$ are independent.\\

To prove Lemma \ref{lemphase11} we first need to show that with a probability close to one, Couplings \eqref{couplage0} and \eqref{couplage12} 
hold during the whole first phase;
to this aim we first prove the following asymptotical result:
\begin{equation}\label{Ttildeinf2eps0}
\liminf_{K \to \infty} \P(A_\eps^K)\geq 1- c \eps
\end{equation}
for a finite $c$ and $\eps$ small enough, where 
\begin{equation}
 \label{defAepsK} A_\eps^K:= \{T_{\eps K}^{(K,1)}\wedge T_{\eps K}^{(K,2)}   \wedge T_{0}^{(K,12)} \leq \tilde{T}^{(K,0)}_\eps\}.
\end{equation}
First applying \eqref{eq2lemmeA} we get the existence of a positive $V$ such that 
for $* \in \{-,+\}$
\begin{equation}\label{defV} \lim_{K \to \infty}\P(\tilde{T}_\eps^{(K,*)}>e^{VK})=1, \end{equation}
where 
\begin{equation} \tilde{T}^{(K,*)}_\eps := \inf \Big\{ t \geq 0, N^{(\eps,*)}_0(t)\notin I_\eps^{(K,0)} \Big\}.  \end{equation}

We divide the probability in \eqref{Ttildeinf2eps0} into two parts according to the position of $\tilde{T}^{(K,0)}_\eps$ 
with respect to $e^{VK}$:
\begin{eqnarray}\label{decoupage1}1 -\P(A_\eps^K) &=&\P(\tilde{T}^{(K,0)}_\eps< T_{\eps K}^{(K,1)}\wedge T_{\eps K}^{(K,2)}   \wedge T_{0}^{(K,12)} ) \\
&=&
\P(\tilde{T}^{(K,0)}_\eps\leq e^{VK},\tilde{T}^{(K,0)}_\eps< T_{\eps K}^{(K,1)}\wedge T_{\eps K}^{(K,2)}   \wedge T_{0}^{(K,12)} ) \nonumber \\
&&+
\P(e^{VK}<\tilde{T}^{(K,0)}_\eps< T_{\eps K}^{(K,1)}\wedge T_{\eps K}^{(K,2)}   \wedge T_{0}^{(K,12)} )\nonumber.
 \end{eqnarray}
Thanks to \eqref{couplage0} and \eqref{defV} we get 
$$ \lim_{K \to \infty}\P(\tilde{T}^{(K,0)}_\eps\leq e^{VK},\tilde{T}^{(K,0)}_\eps< T_{\eps K}^{(K,1)}\wedge T_{\eps K}^{(K,2)}  
\wedge T_{0}^{(K,12)} )=0. $$
Consider now the second probability in \eqref{decoupage1}.
The event $\{e^{VK}<\tilde{T}^{(K,0)}_\eps\}$ means that Couplings \eqref{couplage12} hold at least until time 
$e^{VK} \wedge T_{\eps K}^{(K,1)}\wedge T_{\eps K}^{(K,2)}$. Hence 
\begin{eqnarray*}
\{ e^{VK}<\tilde{T}^{(K,0)}_\eps< T_{\eps K}^{(K,1)}\wedge T_{\eps K}^{(K,2)}   \wedge T_{0}^{(K,12)}\}
&  \subset & \{e^{VK}<(T_0^{(1,-)} \vee T_0^{(2,-)})\wedge T_{\eps K}^{(1,+)} \wedge T_{\eps K}^{(2,+)}\}\\
 & \subset  & \{e^{VK}<T_0^{(1,-)}  \wedge T_{\eps K}^{(1,+)}\}\cup \{e^{VK}<T_0^{(2,-)}\wedge T_{\eps K}^{(2,+)}\}.
 \end{eqnarray*}
But thanks to Equations \eqref{ext_times} and \eqref{equi_hitting} we know that for $i \in \{1,2\}$,
$$ \lim_{K \to \infty}\P(\{e^{VK}< T_0^{(i,-)}\} \bigtriangleup \{T_0^{(i,-)}=\infty\})= 
 \lim_{K \to \infty}\P(\{e^{VK}< T_{\eps K}^{(i,+)}\} \bigtriangleup \{T_0^{(i,+)}<\infty\})=0,$$
where $\bigtriangleup$ denotes the symmetric difference: 
for two sets $B$ and $C$, $B\bigtriangleup C= (B \cap C^c) \cup (C \cap B^c) $.
 This implies that 
$$ \P(e^{VK}<\tilde{T}^{(K,0)}_\eps< T_{\eps K}^{(K,1)}\wedge T_{\eps K}^{(K,2)}   \wedge T_{0}^{(K,12)} )
\leq \sum_{i=1}^{2}
\P(T_0^{(i,-)}=\infty,T_0^{(i,+)}<\infty)+O_K(\eps). $$
But Definitions \eqref{defN1star} and \eqref{defN2star} 
imply that for $i \in \{1,2\}$ the event $\{T_0^{(i,+)}<\infty , T_0^{(i,-)}=\infty\}$ is empty.
This ends the proof of \eqref{Ttildeinf2eps0}. \\
% and adding Equation \eqref{hitting_times} we get 
% \begin{eqnarray*}
%  \P(T_0^{(i,+)}=\infty,T_0^{(i,-)}<\infty) &=& 1 -\P(T_0^{(i,+)}<\infty \sqcup T_0^{(i,-)}=\infty)\\
%  & =& 1 -\P(T_0^{(i,+)}<\infty )-\P(T_0^{(i,-)}=\infty)\\
%  & = & 1 - (1-s_{i0}^{(\eps,+)})-s_{i0}^{(\eps,-)}\leq c \eps,
% \end{eqnarray*}
% where $\sqcup$ stands for the disjoint union, $c$ is finite $c$ and $ \eps $ small enough according to \eqref{hitting_times} and \eqref{diffs+s-}.

We are now able to prove Lemma \ref{lemphase11}. We assume that $A_\eps^K$ holds, which is true with a probability 
close to one according to \eqref{Ttildeinf2eps0}, 
and implies that Coupling \eqref{couplage12} holds on the time interval
$[0,T_{\eps K}^{(K,1)}\wedge T_{\eps K}^{(K,2)}\wedge T_{0}^{(K,12)} ]$.\\

\noindent \textbf{(a):} Let us recall Definitions \eqref{defsi0-} and \eqref{defsi0+}. Thanks to the independence 
of the processes $N_1^{(\eps,+)}$ and $N_2^{(\eps,+)}$ we get by applying \eqref{hitting_times}
\begin{eqnarray*} \P(T_0^{(1,+)}<T_{\eps K}^{(1,+)},T_0^{(2,+)}<T_{\eps K}^{(2,+)})&=&
\left(1-\frac{s_1^{(\eps,+)}}
{1-(1-s_1^{(\eps,+)})^{\lfloor \eps K \rfloor}}\right)
\left( 1-\frac{s_2^{(\eps,+)}}
{1-(1-s_2^{(\eps,+)})^{\lfloor \eps K \rfloor}}\right)\\
& \geq & \Big(1-\frac{S_{10}}{\beta_1}\Big)\Big(1-\frac{S_{20}}{\beta_2}\Big)- c \eps
\end{eqnarray*}
for a finite $c$, $K$ large enough and $\eps$ small enough.
Moreover we have the following inclusion:
$$ \{A_\eps^K, T_0^{(1,+)}<T_{\eps K}^{(1,+)},T_0^{(2,+)}<T_{\eps K}^{(2,+)} \}\subset
\{A_\eps^K, T_0^{(K,1)}<T_{\eps K}^{(K,1)},T_0^{(K,2)}<T_{\eps K}^{(K,2)} \}.$$
We then get
\begin{eqnarray}\label{lowerbounda}
 \P(T_0^{(K,1)}<T_{\eps K}^{(K,1)},T_0^{(K,2)}<T_{\eps K}^{(K,2)}| A_\eps^K)&\geq &
\P(A_\eps^K, T_0^{(1,+)}<T_{\eps K}^{(1,+)},T_0^{(2,+)}<T_{\eps K}^{(2,+)}) \nonumber \\
& \geq & \P(T_0^{(1,+)}<T_{\eps K}^{(1,+)},T_0^{(2,+)}<T_{\eps K}^{(2,+)})-\P((A_\eps^K)^c) \nonumber \\
&\geq &\Big(1-\frac{S_{10}}{\beta_1}\Big)\Big(1-\frac{S_{20}}{\beta_2}\Big)- c \eps,
\end{eqnarray}
for a finite $c$, $K$ large enough and $\eps$ small enough, where we used \eqref{Ttildeinf2eps0} and \eqref{hitting_times}.\\

\noindent \textbf{(b):} The independence 
of the processes $N_1^{(\eps,+)}$ and $N_2^{(\eps,+)}$ again yields
\begin{eqnarray*} \P(T_{\eps K}^{(1,-)}<T_0^{(1,-)},T_0^{(2,+)}<T_{\eps K}^{(2,+)})=
\frac{s_1^{(\eps,-)}}
{1-(1-s_1^{(\eps,-)})^{\lfloor \eps K \rfloor}}
\left( 1-\frac{s_2^{(\eps,+)}}
{1-(1-s_2^{(\eps,+)})^{\lfloor \eps K \rfloor}}\right)
 \geq  \frac{S_{10}}{\beta_1}\Big(1-\frac{S_{20}}{\beta_2}\Big)- c \eps
\end{eqnarray*}
for a finite $c$, $K$ large enough and $\eps$ small enough.
Moreover, thanks to Coupling \eqref{couplage12} we get
\begin{multline*} \{A_\eps^K,T_{\eps K}^{(1,-)}<T_0^{(1,-)},T_0^{(2,+)}<T_{\eps K}^{(2,+)}, T_0^{(2,+)}<T_{\eps K}^{(1,-)}\}\\
\subset 
\{A_\eps^K, T_{\eps K}^{(K,1)}<T_0^{(K,1)},T_0^{(K,2)}<T_{\eps K}^{(K,2)}, T_0^{(K,2)}<T_{\eps K}^{(K,1)} \}. \end{multline*}
But according to Lemma \ref{lembdprocess}, on the event $\{ T_{\eps K}^{(1,-)}<T_0^{(1,-)},T_0^{(2,+)}<T_{\eps K}^{(2,+)} \}$ 
we can find a finite constant $M$ such that for $K$ large enough 
with a probability close to one $T_0^{(2,+)}\leq\alpha \log K+ M$, and $ T_{\eps K}^{(1,-)}$ is close to $\log K/S_{10}^{(\eps,-)}$.
We finally get:
\begin{eqnarray}\label{lowerboundb}
 \P(T_{\eps K}^{(K,1)}<T_0^{(K,1)},T_0^{(K,2)}<T_{\eps K}^{(K,2)})& \geq & 
 \P( T_{\eps K}^{(K,1)}<T_0^{(K,1)},T_0^{(K,2)}<T_{\eps K}^{(K,2)}, T_0^{(K,2)}<T_{\eps K}^{(K,1)}, A_\eps^K) \nonumber \\
 & \geq & \P(T_{\eps K}^{(1,-)}<T_0^{(1,-)},T_0^{(2,+)}<T_{\eps K}^{(2,+)}, T_0^{(2,+)}<T_{\eps K}^{(1,-)})-\P((A_\eps^K)^c) \nonumber \\
 & \geq & \frac{S_{10}}{\beta_1}\Big(1-\frac{S_{20}}{\beta_2}\Big)- c \eps
\end{eqnarray}
for a finite $c$, $K$ large enough and $\eps$ small enough, where we used  \eqref{Ttildeinf2eps0} and  \eqref{hitting_times}.
By interchanging the roles of $1$ and $2$ we derive a lower bound for $(b)$.\\

\noindent \textbf{(c):} Let us now focus on the last inequality. First using again independence 
between $N_1^{(\eps,-)}$ and $N_2^{(\eps,-)}$,
and \eqref{hitting_times} we get
$$ \P(T_{\eps K}^{(1,-)}<T_0^{(1,-)},T_{\eps K}^{(2,-)}<T_0^{(2,-)})\geq 
\frac{S_{10}}{\beta_1}\frac{S_{20}}{\beta_2}- c \eps.$$
But as Coupling \eqref{couplage12} only holds before time $T_{\eps K}^{(K,1)}\wedge T_{\eps K}^{(K,2)} $ we have 
to determine which process, $N_2^K$ or $N_1^K$ hits $\lfloor \eps K \rfloor$ first.
For $i \in \{1,2\}$, from  \eqref{ext_times} and \eqref{equi_hitting},
$$  \P(T_{\eps K}^{(i,-)}<T_0^{(i,-)}, T_0^{(i,-)}<\infty)\leq c \eps. $$
Using again \eqref{equi_hitting} we get the existence of a finite constant $c$ such that:
$$ \P\Big( T_0^{(i,-)}=\infty,
i \in \{1,2\}\Big)\\
\geq 
\frac{S_{10}}{\beta_1}\frac{S_{20}}{\beta_2}+O_K(\eps) $$
and
\begin{multline} \label{tpsinv}  \P\Big(\Big(\mathbf{1}_{\{i=2\}}\alpha +\frac{1-c\eps}{S_{i0}}\Big)\log K \leq T_{\eps K}^{(i,+)}\leq T_{\eps K}^{(i,-)}\leq 
\Big(\mathbf{1}_{\{i=2\}}\alpha +\frac{1+c\eps}{S_{i0}}\Big)\log K,
 T_0^{(i,-)}=\infty,
i \in \{1,2\}\Big)\\
\geq 
\frac{S_{10}}{\beta_1}\frac{S_{20}}{\beta_2}+O_K(\eps). \end{multline}
Hence, as under Assumption \ref{ass1}, $1/S_{10}<1/S_{20}+\alpha$, processes $(N_1^{(\eps,*)}, * \in \{+,-\})$ hit $\lfloor \eps K \rfloor$ before processes
$(N_2^{(\eps,*)}, * \in \{+,-\})$ on the event $\{ T_0^{(i,-)}=\infty , i \in \{1,2\}\}$ with a probability close to one.
The last step consists in determining the values of $(N_2^{(\eps,*)}, * \in \{+,-\})$ on the time
 interval $[(1-c\eps) {\log K}/{S_{10}},(1+c\eps) {\log K}/{S_{10}}]$. First we notice that \eqref{equi_hitting} implies for $* \in \{+,-\}$: 
$$\frac{ T_{K^{S_{20}(\frac{1}{S_{10}}-\alpha*2c\eps)}}^{(2,*)}}{\log K}\to 
\alpha+\frac{S_{20}}{\beta_2 s^{(\eps,*)}_{20}}\Big(\frac{1}{S_{10}}-\alpha*2c\eps\Big), \quad  \text{a.s. on} \ \{ T_0^{(2,*)}=\infty\}. $$
Moreover, we get from \eqref{hitting_times}
\begin{multline*} \P
\Big( T_{K^{S_{20}(\frac{1}{S_{10}}-\alpha-3c\eps)}}^{(2,-)}< T_{K^{S_{20}(\frac{1}{S_{10}}-\alpha-c\eps)}}^{(2,-)}\Big| 
N^{(2,-)}(0)={K^{S_{20}(\frac{1}{S_{10}}-\alpha-2c\eps)}} \Big)\\
=
\frac{(1-s_{20}^{(\eps,-)})^{K^{S_{20}(\frac{1}{S_{10}}-\alpha-c\eps)}}-
(1-s_{20}^{(\eps,-)})^{K^{S_{20}(\frac{1}{S_{10}}-\alpha-2c\eps)}}}{(1-s_{20}^{(\eps,-)})^{K^{S_{20}(\frac{1}{S_{10}}-\alpha-c\eps)}}-
(1-s_{20}^{(\eps,-)})^{K^{S_{20}(\frac{1}{S_{10}}-\alpha-3c\eps)}}}
  \to   0, \quad K \to \infty.
\end{multline*}
This completes the proof of the lower bound for \textbf{(c)}. 
Adding \eqref{lowerbounda} and \eqref{lowerboundb} ends the proof of Lemma \ref{lemphase11}.
\end{proof}

We end the study of the first phase dynamics under Assumptions \ref{asscondinitiale} and 
\ref{ass1} by an approximation of the duration of this phase in the case we are 
interested in:

\begin{lem}\label{lemdureephase1ass1}
Recall the definition of $M_1$ in Lemma \ref{lemphase11}.
Under Assumptions \ref{asscondinitiale} and \ref{ass1},
$$\P\Big((1-c\eps)\frac{\log K}{{S}_{10}}<T^{(K,1)}_{\eps K}<(1+c\eps)\frac{\log K}{{S}_{10}}
\Big|T^{(K,1)}_{\eps K}<\tilde{T}^{(K,0)}_\eps,N_2^K(T^{(K,1)}_{\eps K}) \in J_{M_1\eps}^{(K,2)}\Big)\geq 1+O_K(\eps).$$
\end{lem}

We do not detail the proof of this Lemma. It is a direct consequence of Couplings \eqref{couplage12} and 
Equation \eqref{tpsinv}.

\subsection{Assumption \ref{ass2}}

Under Assumptions \ref{asscondinitiale} and \ref{ass2}, the type $2$ population size 
has a positive probability to become larger than the type $1$ population size during the 
first phase. 
Let us introduce a finite subset of $\N$ which may contain the type $1$ population size at the end of the first phase.
\begin{equation}\label{defJepsK1}
 J_\eps^{(K,1)}:=\Big[K^{S_{10}(\frac{1}{S_{20}}+\alpha-\eps)},
K^{S_{10}(\frac{1}{S_{20}}+\alpha+\eps)}\Big] \cap \N.
\end{equation}

Then we have the following possible states at the end of the first phase.

\begin{lem}\label{lemphase12}
Under Assumptions \ref{asscondinitiale} and \ref{ass2},
there exists a positive constant $M_2$ such that
\begin{eqnarray*}
 &(a)&  \P\Big(T^{(K,i)}_0<T^{(K,i)}_{\eps K}<\tilde{T}^{(K,0)}_\eps, \forall i \in \{1,2\}\Big)=
\Big(1-\frac{S_{10}}{\beta_1}\Big)\Big(1-\frac{S_{20}}{\beta_2}\Big)+O_K(\eps) \nonumber \\
&(b) &  \P\Big(T^{(K,i)}_{\eps K}<T^{(K,i)}_0<\tilde{T}^{(K,0)}_\eps,T^{(K,j)}_0<T^{(K,j)}_{\eps K}<\tilde{T}^{(K,0)}_\eps\Big)=
\frac{S_{i0}}{\beta_i}\Big(1-\frac{S_{j0}}{\beta_j}\Big)+O_K(\eps) ,\quad  i \neq j \in \{1,2\} \nonumber \\
&(c) &  \P\Big(T^{(K,2)}_{\eps K}<\tilde{T}^{(K,0)}_\eps,N_1^K(T^{(K,2)}_{\eps K}) \in J_{M_2\eps}^{(K,1)}\Big)=
\frac{S_{10}}{\beta_1}\frac{S_{20}}{\beta_2}+O_K(\eps).
\end{eqnarray*}
\end{lem}

We do not prove this result as it is very similar to Lemma \ref{lemphase11}. 
The idea is that when the $2$-population survives the first phase, it takes a time of order $\log K/S_{20}$ to hit the value $\lfloor \eps K \rfloor$, 
whereas the $1$-population size needs a time of order $\log K/S_{10}$ to hit such a value. As under Assumption \ref{ass2} $\alpha+1/S_{20}<1/S_{10}$, the 
$2$-population size is the first to represent a positive fraction of the total population size. The value of the $1$-population size at the end of the first phase is 
obtained thanks to Coupling \eqref{couplage12} and Equation \eqref{equi_hitting}.
We have also an equivalent of Lemma \ref{lemdureephase1ass1}:

\begin{lem}\label{lemdureephase1ass2}
Recall the definition of $M_2$ in Lemma \ref{lemphase12}.
Under Assumptions \ref{asscondinitiale} and \ref{ass2},
$$\P\Big((1-c\eps)\Big(\alpha+\frac{\log K}{{S}_{20}}\Big)<T^{(K,2)}_{\eps K}<(1+c\eps)
\Big(\alpha+\frac{\log K}{{S}_{20}}\Big)
\Big|T^{(K,2)}_{\eps K}<\tilde{T}^{(K,0)}_\eps,N_1^K(T^{(K,2)}_{\eps K}) \in J_{M_2\eps}^{(K,1)}\Big)= 1+O_K(\eps).$$
\end{lem}

\subsection{Assumption \ref{ass3}}\label{phase1Ass3}

This case has already been studied in \cite{champagnat2006microscopic} and we said a few words 
about it just after the definition of Assumption \ref{asscondinitiale}. 
We here present rigorously the results of \cite{champagnat2006microscopic} with our notations:

\begin{lem}\label{lemphase13}
Under Assumptions \ref{asscondinitiale} and \ref{ass3},
\begin{eqnarray*}
\P\Big(T^{(K,1)}_0<\tilde{T}^{(K,0)}_\eps\Big)=
\Big(1-\frac{S_{10}}{\beta_1}\Big)+O_K(\eps) ,\quad
\P\Big(T^{(K,1)}_{\eps K}<\tilde{T}^{(K,0)}_\eps\Big)=
\frac{S_{10}}{\beta_1}
+O_K(\eps),
\end{eqnarray*}
and
$$ \P\Big((1-c\eps)\frac{\log K}{{S}_{10}}<T^{(K,1)}_{\eps K}<(1+c\eps)
\frac{\log K}{{S}_{10}}\Big|T^{(K,1)}_{\eps K}<\tilde{T}^{(K,0)}_\eps\Big)= 1+O_K(\eps).$$
\end{lem}

\section{Phases $2$ and $2n$}\label{secondphase}

In this section, we describe the dynamics of "deterministic phases", when  some of the population 
sizes are  well approximated by the solution of a two- or 
three-dimensional competitive Lotka-Volterra system.

\subsection{Two-dimensional case}

Let us denote by $\phi_1^K$ the end of the first phase, when at least one of the mutant population survives:
\begin{equation*}  \phi_1^K:= T_{\eps K}^{(K,1)} \wedge T_{\eps K}^{(K,2)} ,\end{equation*}
by $i(\phi^K_1)$ the label of the first mutant population which hits the value $\lfloor \eps K \rfloor$, and by 
$j(\phi^K_1)$ the label of the other mutant population. 
We will focus on the most interesting case, when the population $j(\phi^K_1)$ does not get 
extinct, as the other ones have already been studied in \cite{champagnat2006microscopic} and \cite{champagnat2011polymorphic}, 
and introduce the event:
\begin{equation}\label{defboncasphase1} B_\eps^K:=\Big\{ T^{(K,i({\phi^K_1}))}_{\eps K}<\tilde{T}^{(K,0)}_\eps,
 N_{j(\phi^K_1)}^K\Big( T^{(K,i({\phi^K_1}))}_{\eps K}\Big) \in J_{M_i\eps}^{(K,j(\phi^K_1))} \Big\}. \end{equation}
We will now consider the second phase of the sweep. 
It corresponds to the interval between the time $T^{(K,i({\phi^K_1}))}_{\eps K}$ when the mutant population $i(\phi^K_1)$ hits the value 
$\lfloor \eps K \rfloor$ and the time when the rescaled population process $(N_0^K,N_{i(\phi^K_1)}^K)/K$ is close enough to the stable equilibrium
of the dynamical system \eqref{S} with labels $0$ and $i(\phi^K_1)$.
To define rigorously the duration of the second phase we need to 
introduce a deterministic time $t_\eps^{(z)}(i,j)$ (see \eqref{deftepsz}) after which the solution of the dynamical system \eqref{S} 
with initial condition $z=(z_i,z_j) \in \R_+^2$ is close to the stable equilibrium.
To do that in a simple way we introduce a notation for the stable equilibrium of \eqref{S} independent of conditions \eqref{defnbara} and 
\eqref{defnbara2}:
\begin{equation}\label{defeqdim2}
 n_{ij}^{(eq)}= n_{ji}^{(eq)}:= 
  \left\{\begin{array}{ll}    \bar{n}_{j}e_j, & \text{if \eqref{defnbara} holds},\\
 \bar{n}_{ij}^{(i)}e_i+\bar{n}_{ij}^{(j)}e_j, 
& \text{if \eqref{defnbara2} holds},
       \end{array}
\right.
\end{equation}
where we recall Definition \eqref{defnij}.
Hence for $\eps>0$, $t_\eps^{(z)}(i,j)$ can be defined by
\begin{equation} \label{deftepsz}
t_\eps^{(z)}(i,j):=
 \inf \{ s \geq 0, \forall t \geq s, \|n_i^{(z)}(t)e_i+n_{j}^{(z)}(t)e_j-n_{ij}^{(eq)}\|\leq \eps^2 \}, 
 \end{equation}
where $\|.\|$ denotes the $L_1$-norm on $\R^2$.
As we do not know precisely the initial value of the rescaled process $(N_0^K,N_{i(\phi^K_1)}^K)/K$ at the beginning of the second phase, 
we consider the supremum over the 
possible $t_\eps^{(z)}(0,i(\phi^K_1))$:
\begin{equation*}
 \mathcal{B}_\eps^K:=\{(z_0,z_{i(\phi^K_1)}),|z_0-\bar{n}_0|\leq 3\eps (C_{0,1}+C_{0,2})/C_{0,0},
 \eps/2 \leq z_i(\phi^K_1) \leq \eps \},
\end{equation*}
and
\begin{equation}\label{defteps}
 t_\eps(0,i(\phi^K_1)):= \sup\{ t_\eps^{(z)}(0,i(\phi^K_1)), z \in \mathcal{B}_\eps^K \},
\end{equation}
which is finite.
We are now able to define rigorously the second phase:
\begin{equation*}
 \text{Phase $2$}:=\Big\{t, T^{(K,i({\phi^K_1}))}_{\eps K} \leq t\leq T^{(K,i({\phi^K_1}))}_{\eps K}+t_\eps(0,i(\phi^K_1)) \Big\}.
\end{equation*}
During the second phase two populations have a size of order $K$, the wild type population 
and the mutant population of type $i(\phi^K_1)$.
We will prove that the dynamics of these two population sizes are well approximated by the deterministic two-dimensional 
competitive Lotka-Volterra system \eqref{S}, which stays in a neighbourhood of its stable equilibrium after the time $t_\eps(0,i(\phi^K_1))$. The duration of the 
second phase,  $t_\eps(0,i(\phi^K_1))$, 
does not tend to infinity with the carrying capacity $K$, unlike the durations of the first and third phases. 
As a consequence, the size of the $j(\phi^K_1)$-population 
stays negligible with respect to $K$ during the second phase.
Recall Definition \eqref{defboncasphase1}. Then we have the following result:

\begin{lem}\label{lemmephase2}
Recall the definitions of the $M_1$ and $M_2$ in Lemmas \ref{lemphase11} and \ref{lemphase12}, respectively.
Then for every $\eps >0$,
\begin{multline*}
  \lim_{K \to \infty}\P \Big(N_{j({\phi^K_1})}^K({\phi_1^K}+s) \in J_{2M_{j(\phi^K_1)}\eps}^K, \forall s \leq t_\eps(0,i(\phi^K_1)) ,\\
 \Big\|\frac{1}{K}(N_{0}^Ke_0+N_{i(\phi^K_1)}^Ke_i)(\phi_1^K+t_\eps(0,i(\phi^K_1)))-n_{0i(\phi^K_1)}^{(eq)}\Big\|\leq \eps^2 \wedge 
\inf_{z \in \mathcal{B}_\eps^K} \frac{n_0^{(z)}(t_\eps(0,i(\phi^K_1)))}{2}
 \Big| B_\eps^K \Big)= 1.
\end{multline*}
\end{lem}

\begin{rem}
 Here we cannot apply directly \eqref{eq2lemmeA}
because it requires positive lower bounds for the jump rates of each population divided by $K$. Indeed in our case, the jump rate of the 
population $2$ is of order 
$$K^{S_{j(\phi_1^K)0}(1/S_{i(\phi_1^K)0}+\alpha(\mathbf{1}_{\{j(\phi_1^K)=1\}}-\mathbf{1}_{\{j(\phi_1^K)=2\}}))}$$ 
which is negligible with respect to $K$. Hence we will couple the process $N_{j(\phi_1^K)}^K$ and the process $(N_0^K,N_{i(\phi_1^K)}^K)$ 
with well known processes to get bounds on their dynamics.
\end{rem}

\begin{proof}[Proof of Lemma \ref{lemmephase2}]
For sake of simplicity we will write $(i,j)$ instead of $(i(\phi^K_1),j(\phi^K_1))$ along the proof.
Let us first prove that during the second phase the process $N_j^K$ does not evolve a lot. 
First, notice that the processes $N_0^K$ and $N_i^K$ are bigger than if they were evolving alone. Hence if we introduce the 
event 
\begin{equation}\label{defCeps}
C_\eps^K:= \Big\{\sup_{s\leq t_\eps(0,i)} \{N_0^K(s)+N_i^K(s)\}>2(\bar{n}_0+\bar{n}_i)K\Big\} ,
\end{equation}
we deduce from \eqref{eq1lemmeA} that,
\begin{equation}\label{limchampN0N1}\lim_{K \to \infty} \sup_{N_0\in I_\eps^{(K,0)}, N_i \in [\eps K-1,\eps K]}
\P_{(N_0,N_i)}(C_\eps^K)=0, \end{equation}
where we used the convention
\begin{equation*} \P_{(N_i,N_j)}(.):=\P(.|(N_i^K,N_j^K)(0)=(N_i,N_j)), \quad (i,j) \in \mathcal{E}^2. \end{equation*}
To control the number of $j$-individuals during the second phase, we introduce two birth and death processes, 
$N_j^{(K,2,-)}$ and $N_j^{(K,2,+)}$ constructed with the same Poisson random measure $Q_j$ as $N_j^K$ (see representation \eqref{defN})
\begin{equation*} \label{defN22-} N_j^{(K,2,-)}(t):= K^{S_{j0}(\frac{1}{S_{i0}}+\alpha(\mathbf{1}_{\{j=1\}}-\mathbf{1}_{\{j=2\}})-M_j\eps)}
-\int_0^t\int_{\R_+}  \mathbf{1}_{\theta\leq (\delta_j+2(\bar{n}_0+\bar{n}_i)(C_{j,0}+C_{j,i}))N_j^{(2,-)}({s^-})} Q_j(ds,d\theta),\end{equation*}
$$ N_j^{(K,2,+)}(t):= K^{S_{j0}(\frac{1}{S_{i0}}+\alpha(\mathbf{1}_{\{j=1\}}-\mathbf{1}_{\{j=2\}})+M_j\eps)}
+\int_0^t\int_{\R_+}  \mathbf{1}_{\theta \leq \beta_jN_j^{(2,+)}(s^-)}Q_j(ds,d\theta).$$
Recall Definitions \eqref{defboncasphase1} and \eqref{defCeps}. We see that on the event $B^K_\eps \cap C_\eps^K$,
\begin{equation*} N_j^{(K,2,-)}(t)\leq N_j^K( \phi_1^K+t )\leq N_j^{(K,2,+)}(t), \quad \text{a.s.} \ \forall t \leq t_\eps(0,i). \end{equation*}
Let us first focus on the process $N_j^{(K,2,-)}$. It is a pure death process with individual death rate:
$$ \delta:=\delta_j+2(\bar{n}_0+\bar{n}_i)(C_{j,0}+C_{j,i}), $$
and we can construct the following martingale associated with this process:
$$M^K_j(t):= N_j^{(K,2,-)}(t)e^{\delta t}-N_j^{(K,2,-)}(0)=-\int_0^t\int_{\R_+}  \mathbf{1}_{\theta\leq \delta N_j^{(K,2,-)}({s^-})}e^{\delta s} \tilde{Q}_j(ds,d\theta), $$
where $\tilde{Q}_j$ is the compensated Poisson measure $\tilde{Q}_j(ds,d\theta):=Q_j(ds,d\theta)-dsd\theta$. Moreover, its quadratic variation 
can be expressed as
$$\langle M_j^K\rangle_t=\int_0^t\int_{\R_+}  \mathbf{1}_{\theta\leq \delta N_j^{(K,2,-)}(s)}e^{2\delta s} dsd\theta=
\int_0^t \delta N_j^{(K,2,-)}(s)e^{2\delta s} ds=
\int_0^t \delta M_j^{K}(s)e^{\delta s} ds. $$
Then Markov Inequality leads to
\begin{multline*}
 \P\Big(N_j^{(K,2,-)}(t_\eps(0,i))<e^{-2\delta t_\eps(0,i)}N_j^{(K,2,-)}(0)\Big)
 = \P\Big(M_j^K(t_\eps(0,i))<\Big(e^{-\delta t_\eps(0,i)}-1\Big)N_j^{(K,2,-)}(0)\Big) \\
\leq \P\Big(\Big(M^K_j(t_\eps(0,i))\Big)^2>\Big(1-e^{-\delta(t_\eps(0,i)}\Big)^2\Big(N_j^{(K,2,-)}(0)\Big)^2\Big)\\
\leq \frac{\E \Big[\langle M_j^K \rangle_{t_\eps(0,i)} \Big]}{\Big(1-e^{-\delta t_\eps(0,i)}\Big)^2\Big(N_j^{(K,2,-)}(0)\Big)^2}
=\frac{\Big(e^{\delta t_\eps(0,i)}-1 \Big)N_j^{(K,2,-)}(0)}{\Big(1-e^{-\delta t_\eps(0,i)}\Big)^2\Big(N_j^{(K,2,-)}(0)\Big)^2},
\end{multline*}
which implies that 
\begin{equation*}
\lim_{K \to \infty} \P\Big(\inf_{s \leq t_\eps(0,i)}\{N_j^{(K,2,-)}(s)\}<e^{-2\delta t_\eps(0,i)}N_j^{(K,2,-)}(0)\Big)=0,
\end{equation*}
as a death process is non increasing. In the same way we prove that 
\begin{equation}\label{majbirthprocess}
\lim_{K \to \infty} \P\Big(\sup_{s \leq t_\eps(0,i)}\{N_j^{(K,2,+)}(s)\}>e^{2\beta_jt_\eps(0,i)}N_j^{(K,2,+)}(0)\Big)=0.
\end{equation}
From \eqref{limchampN0N1} to \eqref{majbirthprocess} we deduce that:
\begin{equation}\label{petitepopphase2}  \lim_{K \to \infty} \ \P \Big(N_{j}^K({\phi_1^K}+s) \in J_{2M_j\eps}^K, \forall s \leq t_\eps(0,i) \Big| B_\eps^K \Big)=1. 
\end{equation}

Now we want to control the dynamics of populations $0$ and $i$ during the second phase. We introduce two pairs of processes, 
$(N_0^{(K,2,-)},N_i^{(K,2,+)})$ and $(N_0^{(K,2,+)},N_i^{(K,2,-)})$ whose dynamics are well know and such that with high probability,
$$ N_k^{(K,2,-)}\leq N_k^K \leq N_k^{(K,2,+)}, \ k \in \{0,i\} ,$$
during the second phase. These processes are defined as follows for $t\geq 0$ and $(a,b)= (i,0)$ or $(0,i)$:
$$ \left\{  \begin{array}{l}
      N_a^{(K,2,-)}(t):=N_a^{K}(\phi_1^K) +\int_0^t\int_{\R_+}  Q_a(ds,d\theta)\Big[ \mathbf{1}_{\theta \leq \beta_aN_a^{(K,2,-)}(s^-)}-\\
      \hspace{1cm}
\mathbf{1}_{0<\theta-\beta_aN_a^{(K,2,-)}(s^-)\leq (\delta_a+C_{a,a}N_a^{(K,2,-)}({s^-})+C_{a,b}N_b^{(K,2,+)}({s^-})
+C_{a,j}K^{S_{j0}(\frac{1}{S_{i0}}+\alpha(\mathbf{1}_{\{j=1\}}-\mathbf{1}_{\{j=2\}})+2M_j\eps)})N_a^{(K,2,-)}({s^-})}\Big],\\
 N_b^{(K,2,+)}(t):=N_b^{K}(\phi_1^K) +\int_0^t\int_{\R_+} Q_b(ds,d\theta)\Big[ \mathbf{1}_{\theta \leq \beta_bN_b^{(K,2,+)}(s^-)}-\\
      \hspace{5cm}
\mathbf{1}_{0<\theta-\beta_bN_b^{(K,2,+)}(s^-)\leq (\delta_b+C_{b,a}N_a^{(K,2,-)}({s^-})+C_{b,b}N_b^{(K,2,+)}({s^-})
)N_b^{(K,2,+)}({s^-})}\Big], 
            \end{array}
  \right. $$
Notice that conditionally on the initial condition of the second phase, $N^K(\phi_1^K)$, 
the processes $(N_k^{(K,2,*)}, k \in \{0,i\}, * \in \{-,+\})$ are independent of the process $N_j^K$, and that on the event 
$$\{\phi_1^K<\infty\} \cap \{N_{j}^K({\phi_1^K}+s) \in J_{2M_j\eps}^K, \forall s \leq t_\eps(0,i)\},$$
we have
\begin{equation*}N_k^{(K,2,-)}(s)\leq N_k^{K}(\phi_1^K+s) \leq N_k^{(K,2,+)}(t),
\  \text{a.s.}\ \forall t \leq t_\eps(0,i)\ \text{and} \  k \in \{0,i\}. \end{equation*}
Moreover, a direct application of \eqref{eq1lemmeA} leads to 
$$ \lim_{K \to \infty} \sup_{(z_0,z_i) \in  I_\eps^{(K,0)}/K\times[\eps/2,\eps]} 
\P(\sup_{0\leq t \leq t_\eps(0,i)}\|(N_0^{(K,2,*)},N_i^{(K,2,\bar{*})})(t)/K-(n_0^{(z)}(t),n_i^{(z)}(t))  \|>\delta) =0, $$
for $\eps$ and $\delta>0$, where $\bar{*}$ denotes the complement of $*$ in $\{-,+\}$, 
$I_\eps^{(K,0)}$ has been defined in \eqref{compact}
and 
$t_\eps(0,i)$ in \eqref{defteps}. Adding \eqref{petitepopphase2} completes the proof 
of Lemma \ref{lemmephase2}.
\end{proof}

The dynamics of the population process is a succession of phases where at least two types of populations have a size of order 
$K$ ($2n$th phases, $n \in \N$) and of phases were at most one population type has a size of order $K$
($(2n+1)$th phases, $n \in \Z_+$). The following lemma completes the description of the dynamics of phases $2n$, 
in the case where two populations have a size of order $K$.
It is a generalization of Lemma \ref{lemmephase2} and we do not give the proof.
Recall Definition \eqref{defteps}. Then

\begin{lem}\label{lemmephase2n2dim}
Let $\{i,j,k\}=\{0,1,2\}$, $\eps>0$, $n,c_i>0$, $\gamma \in (0,1)$ and assume that 
$$(N_i^K(0),N_j^K(0),N_k^K(0))=(\lfloor n K \rfloor,\lfloor \eps K \rfloor,
 \lfloor K^\gamma \rfloor), \ \text{with} \ |n-\bar{n}_i|\leq c_i\eps. $$
Then
\begin{multline*}
  \lim_{K \to \infty} \ \P \Big(N_{k}^K(s) \in [K^{\gamma-\eps},K^{\gamma+\eps}], \forall s \leq t_\eps(i,j) ,\\
 \Big\|\frac{1}{K}(N_{i}^Ke_i+N_{j}^Ke_j)(t_\eps(i,j))-n_{ij}^{(eq)}\Big\|\leq \eps^2 \wedge 
\inf_{z \in \mathcal{B}_\eps^K} \frac{n_i^{(z)}(t_\eps(i,j))}{2} \Big)=1. 
\end{multline*}
\end{lem}

\subsection{Three-dimensional case}\label{section3dim}

From a probabilistic point of view, the case where the three types of populations have a size of order $K$ is simpler, as we can 
approximate the rescaled population process by the three-dimensional deterministic Lotka-Volterra system \eqref{S2}
according to Equation \eqref{eq1lemmeA}.
But the behaviour of the solutions of the three-dimensional Lotka-Volterra systems are much more various than these of the two-dimensional systems.
They have been studied in detail by Zeeman and coauthors \cite{zeeman1993hopf,zeeman1998three,zeeman2003local} and we will now present some of their 
findings.

As in the case of two dimensional systems, the invasion fitnesses 
$$(S_{ij}, S_{ijk}, i,j, \text{ and } k \text{ distinct in } \{0,1,2\})$$ 
will determine the overall behaviour of the flows. In the two dimensional case for interacting populations of types $i$ and $j$,
there are three possibilities:
\begin{enumerate}
 \item[$\bullet$] Either $S_{ij}<0<S_{ji}$; then the only stable fixed point is $\bar{n}_je_j$
 \item[$\bullet$] Or $S_{ij},S_{ji}>0$; then the only stable fixed point is $\bar{n}_{ij}^{(i)}e_i+\bar{n}_{ij}^{(j)}e_j$
 \item[$\bullet$] Or $S_{ij},S_{ji}<0$; then there are two stable fixed points, $\bar{n}_ie_i$ and $\bar{n}_je_j$.
\end{enumerate}

\begin{figure}[h]
\centering
\includegraphics[width=6cm,height=7cm]{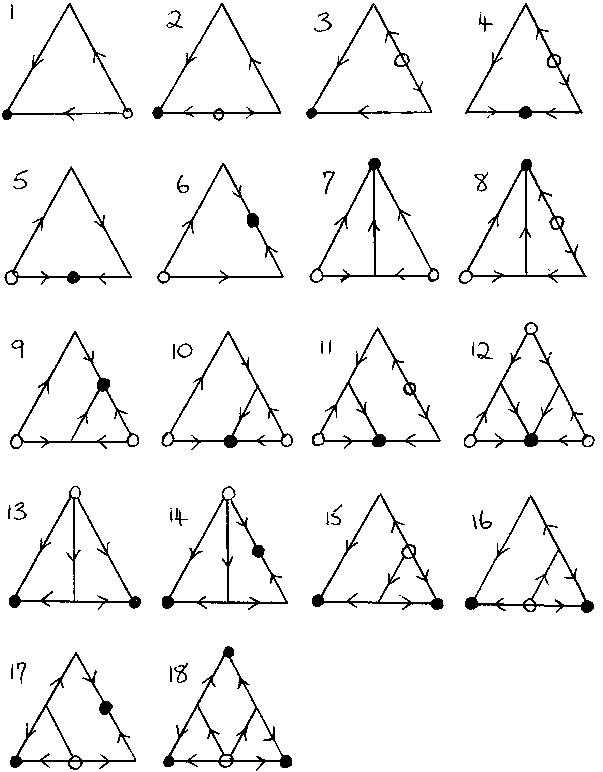}\hspace{1cm}
\includegraphics[width=6cm,height=7cm]{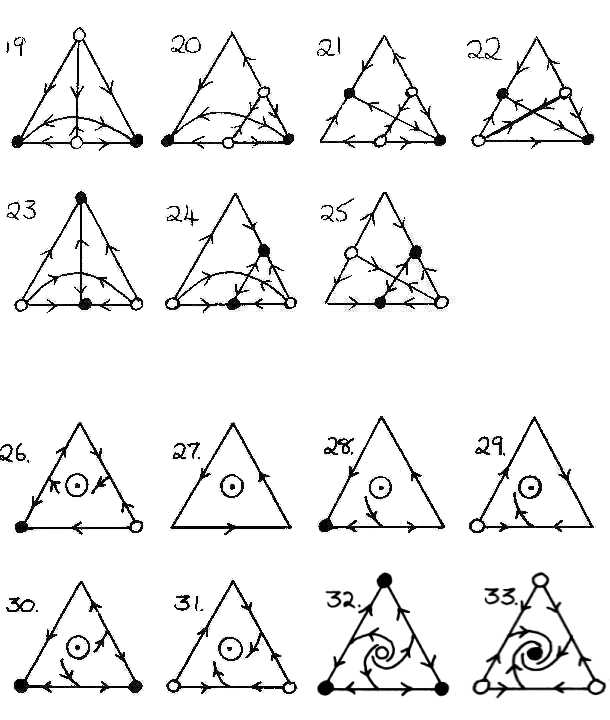}
\caption{The phase portraits on $\Sigma$. A fixed point is represented by a close dot $\bullet$ if it attracts on $\Sigma$; by an open dot 
$\circ$ if it repels on $\Sigma$, and by the intersection of its hyperbolic manifolds if it is a saddle on $\Sigma$. (figures found in \cite{zeeman1993hopf}
and modified to include results of \cite{zeeman1998three} about classes 32 and 33)}
\label{classes1to33}
\end{figure} 

In the three-dimensional case, Zeeman \cite{zeeman1993hopf} numbered 33 equivalence classes.
An ``equivalence class'' is a given combination of invasion fitnesses signs modulo permutation of the indices.
Flows of systems belonging to the same class have the same type of long time behaviour, which is well determined in the 25 first 
classes and more complex in the 8 remaining classes.
These behaviours are represented in Figure \ref{classes1to33} and we will now explain their meaning. 
By an application of Hirsch's Theorem \cite{hirsch1988systems} Zeeman proved that there exists an invariant hypersurface of $\R_+^3$, denoted $\Sigma$, 
such that every non zero trajectory of a three-dimensional competitive Lotka-Volterra system is asymptotic to one in $\Sigma$ 
for large times. $\Sigma$ is called the carrying simplex, beeing a balance between the growth of small populations and the competition 
of large populations. A locally attracting fixed point is represented by a close dot $\bullet$, a locally repelling one by an open dot 
$\circ$, and a saddle by the intersection of its hyperbolic manifolds.

 The classes 1 to 25 have no interior fixed points, 
and their dynamics are 
well known: the system converges to one of the stable fixed points with at most two positive coordinates.
The solutions of systems in class 32 converge to one of the monomorphic equilibrium density $\bar{n}_i$ for $i \in \{0,1,2\}$.
The solutions of systems in class 33 converge to the unique interior fixed points (see \cite{zeeman1998three} for the two last assertions). 
In these classes 26 to 31 the system can have periodic orbits depending on the 
value of the parameters (see Figure \ref{classes26to31}). More precisely Zeeman proved that there exists systems with and without periodic orbits in each of these classes.
We have no general criteria to discriminate between cyclical and converging (or diverging) behaviours of the flows in these classes. Note 
however that Hofbauer and Sigmund (Theorem 15.3.1 in \cite{hofbauer2003evolutionary}) and Zeeman and Zeeman (Theorem 6.7 in \cite{zeeman2003local}) 
provided sufficient conditions to have a global attractor (or repellor).

\begin{figure}[h]
\centering
\includegraphics[width=6cm,height=3cm]{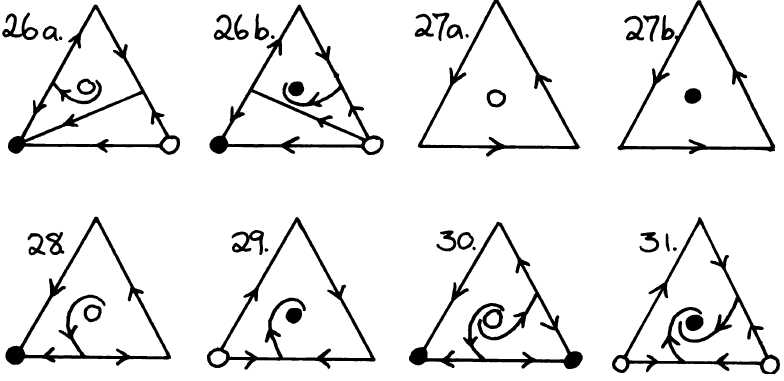}
\caption{Phase portraits in classes 26 to 31 in absence of periodic orbits. (figures found in \cite{zeeman1993hopf})}
\label{classes26to31}
\end{figure}

\section{Phases $3$ and $2n+1$}\label{thirdphase}

At the beginning of the third phase there are two possible cases: either there are already individuals of type $2$ in the population; this 
corresponds to Assumptions \ref{ass1} or \ref{ass2}. Then two types of individuals have sizes close to $\bar{n}_{ij}^{(eq)}K$ 
(defined in \eqref{defeqdim2}), and the last type has a size
of order $K^\gamma$ for some $\gamma \in (0,1)$.
Or the second mutation has not occurred yet (Assumption \ref{ass3}) and there are only individuals of type $0$ and $1$.
Let us first focus on Assumptions \ref{ass1} and \ref{ass2}. 
Then at the beginning of the third phase we have, with a probability close to $S_{10}S_{20}/(\beta_1\beta_2)$, two kinds of initial conditions:
\begin{enumerate}
 \item[$\bullet$] Either $S_{0i(\phi_1^K)}<0<S_{i(\phi_1^K)0}$; 
 then $N_{i(\phi_1^K)}^K$ is close to $\bar{n}_{i(\phi_1^K)}K$, $N_0^K$ 
equals $\lfloor nK\rfloor$ for some
\\
$n \in [\inf_{z \in \mathcal{B}_\eps^K}n_0^{(z)}(t_\eps(0,i(\phi_1^K))),\sup_{z \in \mathcal{B}_\eps^K}n_0^{(z)}(t_\eps(0,i(\phi_1^K)))]$ 
and $N_{j(\phi_1^K)}$ belongs to $J_{2M_{j(\phi_1^K)}\eps}$,
 \item[$\bullet$] Or $S_{0i(\phi_1^K)},S_{i(\phi_1^K)0}>0$; then $(N_0^K,N_{i(\phi_1^K)}^K)$ is close to $(\bar{n}_{0i(\phi_1^K)}^{(0)} K, \bar{n}_{0i(\phi_1^K)}^{(i(\phi_1^K))}K)$ 
and $N_{j(\phi_1^K)}$ belongs to $J_{2M_{j(\phi_1^K)}\eps}$.
\end{enumerate}

In fact 
such initial conditions will also be found in phases $2n+1$, with $n \geq 2$. Indeed we will see that in some cases there will still be two 
populations with sizes of order $K$ and one population with a size of a smaller order after the two first alternations of stochastic 
and deterministic phases (see Figure \ref{RPScycles} for instance).
Lemma \ref{lemmephase2n+1} describes the dynamics of a stochastic phase with such initial conditions. Before stating this lemma, 
we introduce a finite subset of $\N^2$ and a stopping time:

\begin{equation*}I_\eps^{(K,ij)}:=\Big[(N_i,N_j), 
\Big\|\frac{1}{K}(N_ie_i+N_je_j)-(\bar{n}_{ij}^{(i)}e_i+\bar{n}_{ij}^{(j)}e_j) \Big\|
\leq 2\eps \frac{|C_{ij}C_{jk}-C_{ik}C_{jj}| + |C_{ji}C_{ik}-C_{jk}C_{ii}|}{|C_{ii}C_{jj}-C_{ij}C_{ji}|} \Big], \end{equation*}
\begin{equation*}  \tilde{T}^{(K,ij)}_\eps := \inf \Big\{ t \geq 0, (N^K_i(t), N_j^K(t))\notin I_\eps^{(K,ij)} \Big\}.  \end{equation*}

\begin{lem}\label{lemmephase2n+1}
 Let us take $i$, $j$ and $k$ distinct in $\{0,1,2\}$. Assume that $S_{ij}>0$ and $N_k^{K}(0)=\lfloor K^\gamma \rfloor$ 
for some $0<\gamma<1$.
\begin{enumerate}
 \item[$\bullet$] If $S_{ji}<0$, $S_{ki}>0$, $N_i^K(0) \in [(\bar{n}_i-\eps^2)K, (\bar{n}_i+\eps^2)K]$ and 
$N_j^K(0)=\lfloor \delta K \rfloor$ for some $0<\delta<\eps^2$, 
$$ \P\Big( T_0^{(K,j)}<T_{\eps K}^{(K,k)}<\tilde{T}_\eps^{(K,i)} \Big)= 1+O_K(\eps), 
\quad \text{for}\quad \frac{1}{|S_{ji}|}<\frac{1-\gamma}{S_{ki}}, $$
$$
\P\Big( T_{\eps K}^{(K,k)}<\tilde{T}_\eps^{(K,i)} \wedge T_0^{(K,j)}, N_j^K(T_{\eps K}^{(K,k)})
 \in \Big[K^{1-\frac{(1-\gamma+\eps)|S_{ji}|}{S_{ki}}},K^{1-\frac{(1-\gamma-\eps)|S_{ji}|}{S_{ki}}}\Big] \Big)
=1+O_K(\eps), \quad \text{for}\quad \frac{1}{|S_{ji}|}>\frac{1-\gamma}{S_{ki}}.
$$
 \item[$\bullet$] If $S_{ji}>0$, $S_{kij}>0$, 
 $N_l^K(0) \in [(\bar{n}_{ij}^{(l)}-\eps^2)K, (\bar{n}_{ij}^{(l)}+\eps^2)K]$ for $l \in \{i,j\}$, 
 $$ \P\Big( T_{\eps K}^{(K,k)}<\tilde{T}_\eps^{(K,ij)} \Big)= 1+O_K(\eps), $$
 \item[$\bullet$] If $S_{ji}<0$, $S_{ki}<0$, $N_i^K(0) \in [(\bar{n}_i-\eps^2)K, (\bar{n}_i+\eps^2)K]$ and $N_j^K(0)=\lfloor \delta K \rfloor$ for some $0<\delta<\eps^2$, 
 $$ \P\Big( T_0^{(K,j)} \vee T_0^{(K,k)}<\tilde{T}_\eps^{(K,i)} \Big)= 1+O_K(\eps), $$
\item[$\bullet$] If $S_{ji}>0$, $S_{kij}<0$, 
$N_l^K(0) \in [(\bar{n}_{ij}^{(l)}-\eps^2)K, (\bar{n}_{ij}^{(l)}+\eps^2)K]$ for $l \in \{i,j\}$, 
 $$ \P\Big( T_0^{(K,k)}<\tilde{T}_\eps^{(K,ij)} \Big)= 1+O_K(\eps). $$
\end{enumerate}
\end{lem}

The proof of this Lemma is very similar to the proof of Lemma \ref{lemphase11}. It consists in comparing the population process with birth and death processes, 
which can be subcritical in this case, and to show that the equilibrium sizes $\lfloor\bar{n}_iK\rfloor$ or $(\lfloor\bar{n}_{ij}^{(i)}K\lfloor,\rfloor\bar{n}_{ij}^{(j)}K\rfloor)$ are not modified much by 
population(s) with size(s) smaller than $\lfloor\eps K \rfloor$. Hence we do not detail the proof. 
We are also able to get approximation of the stochastic phase duration:

\begin{lem}\label{lemdureephase2n+1ass}
  Let us take $i$, $j$ and $k$ distinct in $\{0,1,2\}$. Assume that $S_{ij}>0$ and $N_k^{K}(0)=\lfloor K^\gamma \rfloor$ 
for some $0<\gamma<1$. Then there exists a finite constant 
$c$ such that,
\begin{enumerate}
 \item[$\bullet$] If $S_{ji}<0$, $S_{ki}>0$, $N_i^K(0) \in [(\bar{n}_i-\eps^2)K, (\bar{n}_i+\eps^2)K]$ and 
$N_j^K(0)=\lfloor \delta K \rfloor$ for some $0<\delta<\eps^2$, 
$$\P\Big((1-c\eps)\frac{1-\gamma}{{S}_{ki}}\log K<T_{\eps K}^{(K,k)}<(1+c\eps)
\frac{1-\gamma}{{S}_{ki}}\log K\Big| T_{\eps K}^{(K,k)}<\tilde{T}_\eps^{(K,i)}\Big)= 1+O_K(\eps).$$
 \item[$\bullet$] If $S_{ji}>0$, $S_{kij}>0$, 
$N_l^K(0) \in [(\bar{n}_{ij}^{(l)}-\eps^2)K, (\bar{n}_{ij}^{(l)}+\eps^2)K]$, $l \in \{i,j\}$, 
$$\P\Big((1-c\eps)\frac{1-\gamma}{{S}_{kij}}\log K<T_{\eps K}^{(K,k)}<(1+c\eps)
\frac{1-\gamma}{{S}_{kij}}\log K\Big| T_{\eps K}^{(K,k)}<\tilde{T}_\eps^{(K,ij)}\Big)= 1+O_K(\eps).$$
 \item[$\bullet$] If $S_{ji}<0$, $S_{ki}<0$, $N_i^K(0) \in [(\bar{n}_i-\eps^2)K, (\bar{n}_i+\eps^2)K]$ and $N_j^K(0)=\lfloor \delta K \rfloor$ for some $0<\delta<\eps^2$, 
\begin{multline*}
\P\Big((1-c\eps)\Big( \frac{1}{|S_{ji}|} \vee \frac{\gamma}{|{S}_{ki}|}\Big)\log K<T_0^{(K,j)} \vee T_0^{(K,k)}<(1+c\eps)
\Big( \frac{1}{|S_{ji}|} \vee \frac{\gamma}{|{S}_{ki}|}\Big)\log K\Big|  T_0^{(K,j)} \vee T_0^{(K,k)}<\tilde{T}_\eps^{(K,i)}\Big)
 \\ = 1+O_K(\eps).
\end{multline*}
\item[$\bullet$] If $S_{ji}>0$, $S_{kij}<0$, 
$N_l^K(0) \in [(\bar{n}_{ij}^{(l)}-\eps^2)K, (\bar{n}_{ij}^{(l)}+\eps^2)K]$ for $l \in \{i,j\}$, 
$$\P\Big((1-c\eps)\frac{1}{|{S}_{kij}|}\log K<T_0^{(K,k)}<(1+c\eps)
\frac{1}{|{S}_{kij}|}\log K\Big|  T_0^{(K,k)}<\tilde{T}_\eps^{(K,ij)} \Big)= 1+O_K(\eps).$$\end{enumerate}
\end{lem}

Let us now describe the dynamics of the third phase when Assumption \ref{ass3} holds. 
As in Lemma \ref{lemmephase2n+1} there are five possibilities, depending on the signs 
of the invasion fitnesses and of the value of $\alpha$.
{To take into account the state of the population at the end of the first phase, we introduce the following notation 
under Assumption \ref{ass3}:
\begin{multline*}
 \P_{\phi_1(N_0,N_1)}(.):=\P(.|(N_0^K,N_1^K)(0)=(N_0,N_1)
\\ \text{ and the second mutant appears at time }
(\alpha-1/S_{10}) \log K)
\end{multline*}
Recall that on the event $\{\phi_1^K < \infty\}$, Lemma \ref{lemphase13} ensures that $\phi_1^K$ is of order $\log K/S_{10}$ 
with a probability close to $1$.
Then we can state the following properties for the population dynamics during the third phase:

\begin{lem}\label{lemmephase3}
 Under Assumptions \ref{asscondinitiale} and \ref{ass3},
\begin{enumerate}
 \item[$\bullet$] If $S_{01}<0$, $S_{21}>0$ and ${1}/{|S_{01}|}<\alpha-{1}/{S_{10}}+1/{S_{21}}$,
$$ \inf_{N_0\leq \eps^2 K, N_1 \in I_\eps^{(K,1)}}\P_{\phi_1(N_0,N_1)}\Big( T_0^{(K,0)}<T_{\eps K}^{(K,2)}<\tilde{T}_\eps^{(K,1)} \Big)
=\frac{S_{21}}{\beta_2}+O_K(\eps), $$
 \item[$\bullet$] If $S_{01}<0$, $S_{21}>0$ and ${1}/{|S_{01}|}>\alpha-{1}/{S_{10}}+1/{S_{21}}$,
\begin{multline*}
  \inf_{\inf_{z \in \mathcal{B}_\eps^K} n_0^{(z)}(t_\eps(0,1))/2 \leq N_0/K\leq \eps^2, N_1 \in I_\eps^{(K,1)}}\P_{\phi_1(N_0,N_1)}
  \Big( T_{\eps K}^{(K,2)}<\tilde{T}_{\eps K}^{(K,1)} \wedge T_0^{(K,0)}, N_0^K(T_\eps^{(K,k)}) \\
 \in \Big[K^{1-|S_{01}|(\alpha-\frac{1}{S_{10}}+\frac{1}{S_{21}}+\eps)},K^{1-|S_{01}|(\alpha-\frac{1}{S_{10}}+\frac{1}{S_{21}}-\eps)}\Big] \Big)
 = \frac{S_{21}}{\beta_2}+O_K(\eps),
\end{multline*}
 \item[$\bullet$] If $S_{01}>0$, $S_{201}>0$, 
 $$ \inf_{\|(N_0,N_1)/K-(\bar{n}_{01}^{(0)},\bar{n}_{01}^{(1)})\|\leq \eps}
 \P_{\phi_1(N_0,N_1)}\Big( T_{\eps K}^{(K,2)}<\tilde{T}_\eps^{(K,01)} \Big)= \frac{S_{201}}{\beta_2}+O_K(\eps), $$
 \item[$\bullet$] If $S_{01}<0$, $S_{21}>0$, 
 $$ \inf_{N_0\leq \eps^2 K, N_1 \in I_\eps^{(K,1)}}\P_{\phi_1(N_0,N_1)}\Big( T_0^{(K,0)} \vee T_0^{(K,2)}<\tilde{T}_\eps^{(K,1)} \Big)
 = 1-\frac{S_{21}}{\beta_2}+O_K(\eps),
 \quad \forall \alpha < \infty, $$
\item[$\bullet$] If $S_{01}>0$, $S_{201}>0$, 
 $$ \inf_{\|(N_0,N_1)/K-(\bar{n}_{01}^{(0)},\bar{n}_{01}^{(1)})\|\leq \eps}
 \P_{\phi_1(N_0,N_1)}\Big( T_0^{(K,2)}<\tilde{T}_\eps^{(K,01)} \Big)= 1-\frac{S_{201}}{\beta_2}+O_K(\eps),\quad \forall \alpha < \infty. $$
\end{enumerate}
\end{lem}}

Once again the proof is very similar to the proof of Lemma \ref{lemphase11}, and we can also derive approximations for the total duration 
of the three third phases:

\begin{lem}\label{lemdureephase2n+1ass2}
  Under Assumptions \ref{asscondinitiale} and \ref{ass3}, there exists a finite constant 
$c$ such that,
\begin{enumerate}
 \item[$\bullet$] If $S_{01}<0$, $S_{21}>0$ and ${1}/{|S_{01}|}<\alpha-{1}/{S_{10}}+1/{S_{21}}$,
$$\P\Big((1-c\eps)\Big(\frac{1}{S_{10}}+\frac{1}{|{S}_{01}|}\Big)\log K<T^{(K,0)}_{0}<(1+c\eps)
\Big(\frac{1}{S_{10}}+\frac{1}{|{S}_{01}|}\Big)\log K\Big|T_0^{(K,0)}<\tilde{T}_\eps^{(K,1)} \wedge  T_\eps^{(K,2)} \Big)= 1+O_K(\eps).$$
 \item[$\bullet$] If $S_{01}<0$, $S_{21}>0$ and ${1}/{|S_{01}|}>\alpha-{1}/{S_{10}}+1/{S_{21}}$,
$$\P\Big((1-c\eps)\Big(\alpha+\frac{1}{{S}_{21}}\Big)\log K<T^{(K,2)}_{\eps K}<(1+c\eps)
\Big(\alpha+\frac{1}{{S}_{21}}\Big)\log K\Big| T_{\eps K}^{(K,2)}<\tilde{T}_\eps^{(K,1)} \wedge T_0^{(K,0)} \Big)= 1+O_K(\eps).$$
 \item[$\bullet$] If $S_{01}>0$, $S_{201}>0$, 
$$\P\Big((1-c\eps)\Big(\alpha+\frac{1}{{S}_{201}}\Big)\log K<T^{(K,2)}_{\eps K}<(1+c\eps)
\Big(\alpha+\frac{1}{{S}_{201}}\Big)\log K\Big| T_{\eps K}^{(K,2)}<\tilde{T}_\eps^{(K,01)} \Big)=1+O_K(\eps). $$
\end{enumerate}
\end{lem}

\section{Proofs of Propositions $1$ to $6$}\label{proofmainth}

We now prove the main results of this paper. The proofs are direct consequences of Lemmas stated in Sections \ref{firstphase}, 
\ref{secondphase} and \ref{thirdphase}. As the proofs of Propositions \ref{propslowdown} and \ref{propspeedup} 
are very similar, we only prove the second one.

\begin{proof}[Proof of Proposition \ref{propspeedup}]
 Let us first consider that Assumptions \ref{asscondinitiale} and \ref{ass1}, and Condition \eqref{cas1speed} are satisfied.\\

\noindent \underline{\textbf{$(1-S_{10}/\beta_1)S_{20}/\beta_2$:}} According to Lemma \ref{lemphase11}, with a probability close to 
$(1-S_{10}/\beta_1)S_{20}/\beta_2$ only the $2$-population survives 
the first phase and hits the size $\lfloor \eps K \rfloor$. Then it 'deterministically' 
competes with the $0$-population and outcompetes it as $S_{02}<0$.
In this case the duration of the invasion, 
$\log K/S_{20}$, is not modified by the presence of the mutant $1$.\\

\noindent \underline{\textbf{$S_{10}S_{20}/(\beta_1\beta_2)$:}} According to Lemma \ref{lemphase11} and \ref{lemdureephase1ass1}, 
with a probability close to 
$S_{10}S_{20}/(\beta_1\beta_2)$ both $1$- and $2$-populations 
survive the first phase, which has a duration close to 
$$ \frac{\log K}{S_{10}}.$$
The $1$-population size is the first to hit $\lfloor \eps K \rfloor$, 
whereas the $2$-population size belongs to $J_{M_1\eps}^{(K,2)}$ (defined in \eqref{defJepsK}) at the end of the first phase.
Lemma \ref{lemmephase2} and Markov Property imply that with a probability close to one,
 the $2$-population size stays of the same order during the second phase, and the 
$0$- and $1$-population sizes become close to $(aK,\bar{n}_1K)$, with $0<a<\eps^2$ depending on ecological parameters but not on $K$, 
at the end of the second phase. As $S_{21}>0$, the $2$-population size grows during the third phase
and takes a time close to 
$$\frac{1}{S_{21}}\Big(1-S_{20}\Big(\frac{1}{S_{10}}-\alpha\Big)\Big)\log K$$
to hit the size $\lfloor \eps K \rfloor$
(Lemma \ref{lemdureephase2n+1ass}). Furthermore at the end of the third phase the $0$-population has a size negligible with respect to $K$ 
(which can be $0$, see Lemma \ref{lemmephase2n+1}). During the fourth phase the $2$-population 'deterministically' 
outcompetes the $1$-population, as $S_{12}<0<S_{21}$ (Lemma \ref{lemmephase2n2dim}). 
Then the $0$- and $1$-populations get extinct, as $S_{02}$ and $S_{12}<0$ (Lemma \ref{lemmephase2n+1}).
Such a trajectory is illustrated in the second simulation of Figure \ref{graphprop22}.
In this case the total duration of the mutant invasion is 
$$ \Big(\frac{1}{S_{10}}+\frac{1}{S_{21}}\Big(1-S_{20}\Big(\frac{1}{S_{10}}-\alpha\Big)\Big)-\alpha\Big)\log K , $$
(the $-\alpha \log K$ is due to the time of the second mutation occurrence)
which has to be compared with the duration of the invasion in the absence of the first mutation, $\log K/S_{20}$.
The resolution of a quadratic equation leads to the condition $S_{20}\in ]S_{10}/(1-\alpha S_{10}),S_{21}[$. More precisely, the inequality 
$$\Big(\frac{1}{S_{10}}+\frac{1}{S_{21}}\Big(1-S_{20}\Big(\frac{1}{S_{10}}-\alpha\Big)\Big)-\alpha\Big)\log K< \frac{\log K}{S_{20}}$$
is equivalent to 
$$ S_{20}^2\frac{1}{S_{21}}\Big( \frac{1}{S_{10}}-\alpha\Big)-S_{20}\Big( \frac{1}{S_{10}}+\frac{1}{S_{21}}-\alpha\Big)+1>0, $$
whose solutions are $S_{20}=S_{21}$ and $S_{20}=S_{10}/(1-\alpha S_{10})$.
As under Assumptions \ref{asscondinitiale} and \ref{ass1}, 
 $$ \frac{1}{S_{20}}> \frac{1}{S_{10}} -\alpha  \Longleftrightarrow S_{20}< \frac{S_{10}}{1-\alpha S_{10}}, $$
 this ends the proof of this case.\\

\noindent \underline{\textbf{$1-S_{20}/\beta_2$:}} Finally with a probability close to $1-S_{20}/\beta_2$ the $2$-population gets extinct during the first phase 
(Lemma \ref{lemphase11}).\\

The proof of the second case follows the same ideas. The only difference is that when the $1$-population survives the first phase, 
the $0$- and $1$-populations 
coexist during the third phase, as $S_{10}$ and $S_{01}>0$. This ends the proof of Proposition \ref{propspeedup}.
\end{proof}

 Again the proofs of Propositions \ref{propdecreaseinvproba} and \ref{propincreaseinvproba} 
are very similar, and we only prove the second one.

\begin{proof}[Proof of Proposition \ref{propincreaseinvproba}]
Let us first suppose that Assumptions \ref{asscondinitiale} and \ref{ass3} hold, and that $S_{20}<0$, $S_{01}>0$, $S_{201}>0$. \\

\noindent \underline{\textbf{$1-S_{10}/\beta_1$:}} According to Lemma \ref{lemphase13}, with a probability close to $1-S_{10}/\beta_1$ the 
mutants $1$ do not survive 
the first phase. 
Then following \eqref{eq2lemmeA} we get that the $0$-population size stays close to its 
equilibrium $\bar{n}_0 K$ during a time of order $e^{VK}$ where $V$ is a positive constant independent of $K$. Hence it is still 
close to this value when the second mutant appears (time $\alpha \log K$).
The latter one has a negligible probability to survive as the invasion fitness $S_{20}$ is negative.
After the $2$-population extinction, 
using again \eqref{eq2lemmeA}, we get that
the $0$-population size stays close to its equilibrium value $\bar{n}_0K$ during a time 
larger than $e^{VK}$ where $V$ is a positive constant independent of $K$. \\

\noindent \underline{\textbf{$S_{10}/\beta_1(1-S_{201}/\beta_2)$:}} According to Lemma \ref{lemphase13}, with a probability close to $S_{10}/\beta_1$ the 
mutants $1$ survive 
the first phase, and the $0$- and $1$-population sizes get close to their coexisting equilibrium 
$(\lfloor\bar{n}_{01}^{(0)}K\rfloor,\lfloor\bar{n}_{01}^{(1)}K\rfloor)$ during the second phase. 
Then we get from \eqref{eq2lemmeA} that the $0$- and $1$-population sizes stay close to this 
coexisting equilibrium 
 during a time larger than $e^{VK}$ for large $K$ where $V$ is a positive constant independent of $K$. Hence they are still close to 
 this state 
 when the second mutant appears (time $\alpha \log K$). 
 The latter one has a probability close to $1-S_{201}/\beta_2$ to get extinct before hitting  $\lfloor \eps K\rfloor$ 
 (Lemma \ref{lemmephase3}).
 After the $2$-population extinction, the $0$- and $1$-population sizes stay close to their equilibrium value 
 $\lfloor \bar{n}_{01}^{(0)}K\rfloor$ and $\lfloor\bar{n}_{01}^{(1)}K\rfloor$ during a time 
larger than $e^{VK}$ where $V$ is a positive constant independent of $K$ (Equation \eqref{eq2lemmeA}).\\
 
\noindent \underline{\textbf{$S_{10}S_{201}/(\beta_1\beta_2)$:}} Applying again Lemma \ref{lemphase13} and Equation \eqref{eq2lemmeA}, 
we get that with a probability close to $S_{10}/\beta_1$ the $0$- and $1$-population sizes get close to $(\bar{n}_{01}^{(0)}K,\bar{n}_{01}^{(1)}K)$
 and are still in this configuration when the second mutant occurs.  
 The latter one has a probability close to $S_{201}/\beta_2$ to hit a size of order 
 $K$ (Lemma \ref{lemmephase3}) whereas $(N_0,N_1)$ still belongs to $I_\eps^{(K,01)}$. 
 Then the final state depends on the signs of the other invasion fitnesses. The approximating deterministic Lotka-Volterra system after the third 
 phase can belong to 7-12, 29, 31 or 33ird class described by Zeeman (see Figures \ref{classes1to33} and \ref{tableass2} case  
%  $\tikz \node[draw,circle,minimum size=.02cm, node distance=10.75cm]{B};$
 \ovalbox{B}).
 In all the cases the density of $2$-type individuals does not tend to $0$ during a time 
larger than $e^{VK}$ where $V$ is a positive constant independent of $K$.\\

 The proof of the second case follows the same ideas. We just have to take into account the value of $\alpha \log K$ 
 compared to some invasion time 
 to know whether the $0$-population gets extinct before the hitting of $\lfloor \eps K\rfloor$ by the $2$-population. We detail this kind of 
comparison in the proof of Proposition \ref{prolizards} and end here the proof of Proposition \ref{propdecreaseinvproba}.
\end{proof}

Let us conclude this section with the proof of Proposition \ref{prolizards}. The proof of Proposition \ref{proannulationeffet} 
follows the same outline and 
is simpler; hence we leave it to the reader.

\begin{proof}[Proof of Proposition \ref{prolizards}]
 Suppose that Assumptions \ref{asscondinitiale}, \ref{ass3} and Conditions \eqref{cond1RPS} and \eqref{cond2RPS} hold.\\

 \noindent \underline{\textbf{$1-S_{10}/\beta_1$:}} See the beginning of the proof of Proposition \ref{propdecreaseinvproba}. \\
 
 \noindent \underline{\textbf{$(S_{10}/\beta_1)(1-S_{21}/\beta_2)$:}} Lemma \ref{lemphase13} and Equation \eqref{eq1lemmeA}
 imply that with a probability close to $S_{10}/\beta_1$, 
 the $1$-population survives the first phase and the population state at the end of the second phase satisfies 
$ (N_0,N_1)(\phi_1^K)=(\lfloor n_0K\rfloor, \lfloor n_1 K \rfloor) $ with 
$$\inf_{z \in \mathcal{B}_\eps^K} \frac{n_0^{(z)}(t_\eps(0,1))}{2}<n_0<\eps^2\quad \text{and} \quad|n_1-\bar{n}_1|\leq \eps,$$
where $t_\eps(0,1)$ has been defined in \eqref{defteps}.
Then with a probability close to $1-S_{21}/\beta_2$ the $2$-population does not survive, and the $0$-population size, which can be compared 
to a subcritical birth and death process as $S_{01}<0$, also hits $0$ while the $1$-population size is still close to 
$\lfloor\bar{n}_1K \rfloor$ (Lemma \ref{lemmephase3}).
In this case, the $1$-population size stays close to its equilibrium value $\lfloor\bar{n}_1K\rfloor$ during a time 
larger than $e^{VK}$ where $V$ is a positive constant independent of $K$ (Equation \eqref{eq2lemmeA}).\\

 \noindent \underline{\textbf{$S_{10}S_{21}/(\beta_1\beta_2)$:}} {First applying again Lemma \ref{lemphase13} 
 and Equation \eqref{eq1lemmeA} we get that w}ith a probability close to $S_{10}/\beta_1$, 
 the $1$-population survives the first phase and the population state at the end of the second phase satisfies 
$ (N_0,N_1)(\phi_1^K)=(\lfloor n_0K\rfloor, \lfloor n_1 K \rfloor) $ with 
$$\inf_{z \in \mathcal{B}_\eps^K} \frac{n_0^{(z)}(t_\eps(0,1))}{2}<n_0<\eps^2\quad  \text{and} \quad |n_1-\bar{n}_1|\leq \eps.$$
Then Lemma \ref{lemmephase3} implies that with a probability close to $S_{21}/\beta_2$, the $2$-population size hits the 
value $\lfloor \eps K \rfloor $ before the extinction of the $0$-population, because 
$$\Big(\frac{1}{S_{10}}+\frac{1}{|S_{01}|}\Big)\log K,$$
which is the approximate 
extinction time of the 
$0$-population (Lemma \ref{lemdureephase2n+1ass2}), is bigger than 
$$\Big(\alpha + \frac{1}{S_{21}}\Big)\log K,$$
which is the approximate hitting time of $\lfloor \eps K \rfloor$ 
by the $2$-population size (again Lemma \ref{lemdureephase2n+1ass2}). 
Combining Lemmas \ref{lemphase13}, \ref{lemmephase3} and \ref{lemdureephase2n+1ass2} we 
even get that at the end of the third phase, the $0$-population 
size is of order 
$$K^{1-|S_{01}|(\alpha - \frac{1}{S_{10}}+\frac{1}{S_{21}})},$$
and that the duration of the third phase is 
\begin{equation}\label{dureethird} \Big( \alpha - \frac{1}{S_{10}}+\frac{1}{S_{21}} \Big)\log K. \end{equation}
During the fourth phase we can approximate the dynamics of the $1$- and $2$-populations, which have a size of order $K$, by 
the system \eqref{S} with $(i,j)=(1,2)$
(Lemma \ref{lemmephase2n2dim}). As $S_{12}<0<S_{21}$,
the $1$- and $2$-population states at the end of the fourth phase are 
$ (\lfloor n_1K\rfloor, \lfloor n_2 K \rfloor) $ with 
$$\inf_{z \in \mathcal{B}_\eps^K} \frac{n_1^{(z)}(t_\eps(1,2))}{2}<n_1<\eps^2\quad  \text{and} \quad |n_2-\bar{n}_2|\leq \eps.$$
During the fifth phase, the $0$-population size has an evolution comparable to this of a supercritical birth and death process 
($S_{02}>0$) and takes a time 
\begin{equation} \label{time1} \frac{|S_{01}|}{S_{02}}\Big(\alpha -\frac{1}{S_{10}}+\frac{1}{S_{21}} \Big)\log K \end{equation}
to hit $\lfloor \eps K\rfloor$ (Lemmas \ref{lemmephase2n+1} and \ref{lemdureephase2n+1ass}). The $1$-population size has an evolution comparable to this of a subcritical birth and death process 
($S_{12}<0$) and takes a time
\begin{equation}\label{time2} \frac{\log K}{|S_{12}|} \end{equation}
to get extinct. As according to Condition \eqref{cond2RPS}
$$  \frac{|S_{01}|}{S_{02}}\Big(\alpha -\frac{1}{S_{10}}+\frac{1}{S_{21}} \Big)<\frac{1}{|S_{12}|}, $$
\eqref{time1} and \eqref{time2} imply that the $0$-population size hits $\lfloor \eps K \rfloor$ before the extinction 
of the $1$-population, the fifth phase has a duration 
of order 
\begin{equation}\label{dureefifth} \frac{|S_{01}|}{S_{02}}\Big(\alpha -\frac{1}{S_{10}}+\frac{1}{S_{21}} \Big)\log K, \end{equation}
and at the end of the fifth phase, the $1$-population size is of order
$$ K^{1-\frac{|S_{12}||S_{01}|}{S_{02}}(\alpha - \frac{1}{S_{10}}+\frac{1}{S_{21}})} .$$
We then again apply the same reasoning: during the sixth phase, the $0$-population outcompetes the $2$-population. 
During the seventh phase, which has a duration of order
\begin{equation}\label{dureeseven} \frac{|S_{12}||S_{01}|}{S_{10}S_{02}}\Big(\alpha -\frac{1}{S_{10}}+\frac{1}{S_{21}} \Big)\log K, \end{equation}
the $1$-population size hits the value $\lfloor \eps K \rfloor$ and the $2$-population sizes ends at a value of order
$$ K^{1-\frac{|S_{20}||S_{12}||S_{01}|}{S_{10}S_{02}}(\alpha - \frac{1}{S_{10}}+\frac{1}{S_{21}})} ,$$
where we used Condition \eqref{cond2RPS} which implies that
$$ \frac{|S_{12}||S_{01}|}{S_{10}S_{02}}\Big(\alpha -\frac{1}{S_{10}}+\frac{1}{S_{21}} \Big)<\frac{1}{|S_{20}|}. $$
Adding the durations of the third, fifth and seventh phases in \eqref{dureethird}, \eqref{dureefifth} and \eqref{dureeseven}, we get the 
duration of the first cycle,
$$ \Big(\alpha -\frac{1}{S_{10}}+\frac{1}{S_{21}} \Big) \Big(1+\frac{|S_{01}|}{S_{02}} +\frac{|S_{01}||S_{12}|}{S_{02}S_{10}} \Big)
\log K .$$
Then by induction we prove that at the end of the phase $3+l$, $l \in \N$, the $0$-population size is of order
$$K^{1-|S_{01}|(\alpha - \frac{1}{S_{10}}+\frac{1}{S_{21}})(\frac{|S_{12}||S_{01}||S_{20}|}{S_{10}S_{02}S_{21}})^l},$$
at the end of the phase $5+l$, $l \in \N$, the $1$-population size is of order
$$K^{1-\frac{|S_{12}||S_{01}|}{S_{02}}(\alpha - \frac{1}{S_{10}}+\frac{1}{S_{21}})(\frac{|S_{12}||S_{01}||S_{20}|}{S_{10}S_{02}S_{21}})^l},$$
and at the end of the phase $7+l$, $l \in \N$, the $2$-population size is of order
$$K^{1-\frac{|S_{20}||S_{12}||S_{01}|}{S_{10}S_{02}}(\alpha - \frac{1}{S_{10}}+\frac{1}{S_{21}})(\frac{|S_{12}||S_{01}||S_{20}|}{S_{10}S_{02}S_{21}})^l}.$$
We also prove that the $3+l$th phase, $l \in \N$, has a duration close to
\begin{equation*} \Big(\frac{|S_{12}||S_{01}||S_{20}|}{S_{10}S_{02}S_{21}}\Big)^l\Big( \alpha - \frac{1}{S_{10}}+\frac{1}{S_{21}} \Big)\log K, \end{equation*}
the $5+l$th phase, $l \in \N$, has a duration close to
\begin{equation*} \Big(\frac{|S_{12}||S_{01}||S_{20}|}{S_{10}S_{02}S_{21}}\Big)^l\frac{|S_{01}|}{S_{02}}\Big(\alpha -\frac{1}{S_{10}}+\frac{1}{S_{21}} \Big)\log K, \end{equation*}
and the $7+l$th phase, $l \in \N$, has a duration close to
\begin{equation*} \Big(\frac{|S_{12}||S_{01}||S_{20}|}{S_{10}S_{02}S_{21}}\Big)^l \frac{|S_{12}||S_{01}|}{S_{10}S_{02}}\Big(\alpha -\frac{1}{S_{10}}+\frac{1}{S_{21}} \Big)\log K. \end{equation*}
This completes the proof of Proposition \ref{prolizards}.
\end{proof}

\appendix

\section{Technical results}\label{knownresults}

This section is dedicated to technical results needed in the proofs.
We first recall some facts about birth and death processes. They are used in Sections \ref{firstphase} to 
\ref{thirdphase} and can be found in \cite{MR2047480}:

\begin{lem}\label{lembdprocess}
Let $Z=(Z_t)_{t \geq 0}$ be a birth and death process with individual birth and death rates $b$ and $d $. For $a \in \R_+$, 
$T_a=\inf\{ t\geq 0, Z_t=\lfloor a \rfloor \}$ and $\P_i$ (resp. $\E_i$) is the law (resp. expectation) of $Z$ when $Z_0=i \in \N$. Then 
\begin{enumerate} 
  \item[$\bullet$] For $(i,j,k) \in \Z_+^3$ such that $j \in (i,k)$,
 \begin{equation} \label{hitting_times} \P_j(T_k<T_i)=\frac{1-(d/b)^{j-i}}{1-(d/b)^{k-i}} .\end{equation}
 \item[$\bullet$] If $0<d \neq b$, for every $i\in \Z_+$ and $t \geq 0$,
\begin{equation} \label{ext_times} \P_{i}(T_0\leq t )= \Big( \frac{d(1-e^{(d-b)t})}{b-de^{(d-b)t}} \Big)^{i}.\end{equation}
 \item[$\bullet$] If $0<d<b$, on the non-extinction event of $Z$, which has a probability $1-(d/b)^{Z_0}$, the following convergence holds:
\begin{equation} \label{equi_hitting}  T_N/\log N \underset{N \to \infty}{\to} (b-d)^{-1},\quad a.s.  \end{equation}
\end{enumerate}
\end{lem}

We also need large deviation results to quantify the difference between the rescaled population process $N^K/K$ and the approximating 
deterministic Lotka-Volterra processes \eqref{S} and \eqref{S2}. The following statements can be found in \cite{champagnat2006microscopic} Theorem 3 (b) and (c) and 
in \cite{champagnat2011polymorphic} Proposition A.2. They follow from 
Dupuis and Ellis \cite{dupuisweak} (Theorem 10.2.6 in Chapter 10):

\begin{lem}
Let $C$ be a compact of $(\R_+^*)^2\times \{0\} $ or $(\R_+^*)^3$, and $T$ a finite positive constant. Then for every positive $\delta$,
\begin{equation}\label{eq1lemmeA}
  \lim_{K \to \infty} \sup_{z \in C} 
\P\Big(\sup_{ t \leq T}\|N^{K}(t)/K-n^{(z)}(t) \|>\delta \Big|N_0^K(0)=\lfloor zK \rfloor\Big) =0,\end{equation}
where $\lfloor zK\rfloor = (\lfloor z_0K\rfloor,\lfloor z_1K\rfloor,\lfloor z_2K\rfloor) $, and $n^{(z)}$ is the solution of 
\eqref{S2} with initial condition $z$.

Let $\bar{n}$ denote a stable equilibrium of a competitive Lotka-Volterra system in dimension one, two or 
three, with all coordinates positive. Let $\eps>0$ and $N^K$ denote the population process with the same ecological 
parameters as the considered Lotka-Volterra system and carrying capacity $K$. Then there exists a positive constant $V$ such that 
\begin{equation}\label{eq2lemmeA}
 \lim_{K \to \infty}\P\Big(\sup_{t \leq e^{VK}}
\Big\| \frac{N^K(t)}{K}-\bar{n} \Big\|\leq 2\eps \Big| \ \Big\| \frac{N^K(0)}{K}-\bar{n} \Big\|\leq \eps \Big)=1. \end{equation}
\end{lem}

Let us now prove Lemma \ref{transivsnontransi} which precises the conditions needed to have transitive interactions 
between several types of individuals:

\begin{proof}[Proof of Lemma \ref{transivsnontransi}]
 Let $i,j,k$ be in $\mathcal{E}$ and recall notation \eqref{tildeC}. The relations
$$ i \prec j \quad \text{and} \quad j \prec k $$
 are equivalent to
 \begin{equation*}
  \frac{\rho_i}{\rho_j}- \tilde{C}_{ij}<0< \frac{\rho_j}{\rho_i}- \tilde{C}_{ji}\quad \text{and} \quad 
   \frac{\rho_j}{\rho_k}- \tilde{C}_{jk}<0< \frac{\rho_k}{\rho_j}- \tilde{C}_{kj},
 \end{equation*}
 or in other terms 
 \begin{equation*}
  \frac{\rho_i}{\rho_j}< \tilde{C}_{ij} \wedge \frac{1}{\tilde{C}_{ji}}\quad \text{and} \quad 
   \frac{\rho_j}{\rho_k}< \tilde{C}_{jk}\wedge \frac{1}{\tilde{C}_{kj}}.
 \end{equation*}
 \begin{enumerate}
  \item If \eqref{transicas1} holds, then 
  \begin{eqnarray*} S_{ik}&=& \rho_k \Big( \frac{\rho_i}{\rho_k}- \tilde{C}_{ik} \Big)= \rho_k \Big( \frac{\rho_i}{\rho_j}\frac{\rho_j}{\rho_k}- \tilde{C}_{ik} \Big)\\
  &< & \rho_k \Big( \Big(\tilde{C}_{ij}\wedge \frac{1}{\tilde{C}_{ji}}\Big) 
  \Big(\tilde{C}_{jk}\wedge \frac{1}{\tilde{C}_{kj}}\Big)- \tilde{C}_{ik} \Big)
   \leq  \rho_k \Big(\Big(C_2 \wedge \frac{1}{C_1}\Big)^2 - C_1\Big)\leq 0,
  \end{eqnarray*}
and
  \begin{eqnarray*} S_{ki}&=& \rho_i \Big( \frac{\rho_k}{\rho_i}- \tilde{C}_{ki} \Big)= 
  \rho_i \Big( \frac{\rho_k}{\rho_j}\frac{\rho_j}{\rho_i}- \tilde{C}_{ki} \Big)\\
  &> & \rho_i \Big( \Big(\tilde{C}_{kj}\vee \frac{1}{\tilde{C}_{jk}} \Big) \Big(\tilde{C}_{ji} \vee \frac{1}{\tilde{C}_{ij}} \Big)- \tilde{C}_{ki} \Big)
   \geq  \rho_k \Big(\Big(C_1 \vee \frac{1}{C_2}\Big)^2 - C_2\Big)\geq 0,
  \end{eqnarray*}
which implies that 
$$i \prec k.$$
  \item If \eqref{transicas2} holds, then 
one of the two previous inequalities is not satisfied. Hence either $i \prec k$ or $i = k$.
 \item Let us now assume that \eqref{transicas2} holds. We can choose $\eta>0$ such that 
 $$ \Big[\Big( C_2 \wedge \frac{1}{C_1} \Big)- \eta \Big]^2>C_1 \quad \text{and} \quad 
 \Big[\Big( C_2 \wedge \frac{1}{C_1} \Big)- \eta \Big]^{-2}<C_2.$$
 Assume now that 
 $$ \frac{\rho_i}{\rho_j}= \frac{\rho_j}{\rho_k}= \Big( C_2 \wedge \frac{1}{C_1} \Big)- \eta .$$
Then
 \begin{equation*}
  \frac{\rho_i}{\rho_j}- C_{2}<0< \frac{\rho_j}{\rho_i}- C_{1}, \quad 
\frac{\rho_j}{\rho_k}- C_{2}<0< \frac{\rho_k}{\rho_j}- C_{1}
  \quad \text{and} \quad 
   \frac{\rho_i}{\rho_k}- {C}_{1}>0> \frac{\rho_k}{\rho_i}- {C}_{2}.
 \end{equation*}
\end{enumerate}
This ends the proof of Lemma \ref{transivsnontransi}.
\end{proof}

We now present two examples of three dimensional Lotka-Volterra systems satisfying conditions of Proposition 
\ref{prolizards} and exhibiting different long time behaviours. In the first one, the solutions converge to a limit cycle 
with constant period whereas in the second one they converge to the unique globally attracting fixed point of the system.
Before these two examples, we need to give a definition and recall a result of Hofbauer and Sigmund.

\begin{defi}
 A matrix $A=(a_{ij})_{0\leq i ,j \leq 2} \in M_3(\R)$ is called Volterra-Lyapunov stable if there exist positive numbers $d_0,d_1,d_2$ such that
 $$ \underset{0\leq i,j \leq 2}{\sum} d_ia_{ij}n_in_j<0, \quad \forall n \neq 0. $$
\end{defi}

Let us introduce the matrix 
\[ A=\left( \begin{array}{ccc}
C_{0,0} & C_{0,1} & C_{0,2} \\
C_{1,0} & C_{1,1} & C_{1,2} \\
C_{2,0} & C_{2,1} & C_{2,2} \end{array} \right).\] 
Then we have:

\begin{theo}[Theorem 15.3.1 in \cite{hofbauer2003evolutionary}]\label{theoHS}
 If $-A$ is Volterra-Lyapunov stable then the Lotka-Volterra system \eqref{S2} has one globally stable fixed point.
\end{theo}

This theorem  
is needed to construct the second example in the following illustration:

\begin{proof}[Illustration of Remark \ref{remarkRPS}]
\begin{enumerate}
 \item In \cite{neumann2007continuous}, the authors model the RPS interactions of escherichia coli strains by the following 
system:
\begin{equation*}
\left\{\begin{array}{ll}
	\dot{n}_0=(\delta-\kappa_1n_0-\mu n_1-\mu n_2)n_0,\\
        \dot{n}_1=(\eta-\mu n_0-\kappa_3n_1-\mu n_2)n_1,\\
	\dot{n}_2=(\beta-(\mu+\gamma)n_0-\mu n_1-\mu n_2)n_2,
       \end{array}
\right.
 \end{equation*}
where all the parameters belong to $\R_+^*$. They
give an example for which the solutions converge to a limit cycle with constant period.
It corresponds to the following values of the parameters:
$$ \eta = \mu = 1, \quad   \kappa_1= \beta=2, \quad \kappa_2=1,75, \quad \gamma = 2,84, \quad 
\kappa_3= 0,844.. \quad \text{and} \quad  \delta = 1,156... $$
For these values of the parameters we have:
$$ S_{10}= 0,422..., \quad S_{01}= - 0,0288... $$
$$ S_{21}= 0,816..., \quad S_{12}= - 0,0143 $$
$$ S_{02}= 0,0131..., \quad S_{20}= - 0,220.... $$
These invasion fitnesses satisfy \eqref{cond1RPS} and 
$$ S_{21}>|S_{01}|, \quad S_{02}S_{21}>|S_{01}||S_{12}|\quad \text{and} \quad S_{02}S_{10}S_{21}>|S_{01}||S_{12}||S_{20}|  .$$
Hence we can choose $\alpha$ such that \eqref{cond2RPS} holds.

\item Recalling the definition of invasion fitnesses in \eqref{deffitinv} we get that $-A$ is 
Volterra-Lyapunov stable if there exist positive constants $d_0,d_1, d_2$ such that for every $n \neq 0$,
 $$ \underset{0\leq i,j \leq 2}{\sum} \frac{(\beta_i -\delta_i)d_i}{C_{i,i}}(\beta_i -\delta_i-S_{ij})
 \frac{C_{i,i}n_i}{\beta_i -\delta_i}\frac{C_{j,j}n_j}{\beta_j -\delta_j}>0, $$
which amounts to the existence of
positive constants $d_0,d_1, d_2$ such that for every $n \neq 0$,
 $$ \underset{0\leq i,j \leq 2}{\sum} d_i(\beta_i -\delta_i-S_{ij})
 n_in_j>0. $$
Expanding the sum and using the notation $\rho_i=\beta_i-\delta_i$ for $i \in \mathcal{E}$ we get the condition
\begin{multline*} 
 d_0\rho_0n_0^2+ d_1\rho_1n_1^2+ d_2\rho_2n_2^2+ 
 \Big[ d_0(\rho_0+|S_{01}|)+ d_1(\rho_1-S_{10}) \Big]n_0n_1+\\
 \Big[ d_0(\rho_0-S_{02})+ d_2(\rho_2+|S_{20}|) \Big]n_0n_2+
  \Big[ d_1(\rho_1+|S_{12}|)+ d_2(\rho_2-S_{21}) \Big]n_1n_2>0.
 \end{multline*}
Let us now make the following choice:
\begin{equation}\label{choixparaex2} \rho_i d_i =1, \quad i \in \mathcal{E}, 
\quad \frac{\rho_0-S_{02}}{\rho_0}=\frac{\rho_1-S_{10}}{\rho_1}=\frac{\rho_2-S_{21}}{\rho_1}=\eta, \quad \frac{|S_{01}|}{\rho_0}=\frac{|S_{12}|}{\rho_1}=
\frac{|S_{20}|}{\rho_2}=\eta, \end{equation}
where $0<\eta<1/2$.
The condition to satisfy becomes
\begin{multline*} 
 n_0^2+ n_1^2+ n_2^2+ \frac{1+2\eta}{2}(2n_0n_1+2n_0n_2+2n_1n_2)
 =  \frac{1-2\eta}{2}(n_0^2+ n_1^2+ n_2^2)+\frac{1+2\eta}{2}(n_0+n_1+n_2)^2 >0,
 \end{multline*}
which holds for every non null $n \in \R^3$.
The end of the proof consists in noticing that we can indeed choose the parameters as in \eqref{choixparaex2}. Assume for example that 
the $\rho_i$'s and the $C_{i,i}$'s are given.
Then it is enough to take
$$ C_{i+1,i}=\eta \frac{\rho_{i+1}}{\rho_i}C_{i,i},  \quad  C_{i,i+1}=(1+\eta) \frac{\rho_i}{\rho_{i+1}}C_{i+1,i+1}, \quad 
\text{where $i$ is modulo $2$}.  $$
 Finally it is easy to check that the conditions of Proposition \ref{prolizards} are satisfied.
 Applying \ref{theoHS} we get that the Lotka-Volterra system \eqref{S2} has one globally stable fixed point. This ends the construction of the second example.
 \end{enumerate}
\end{proof}

\section{Complete description of possible dynamics}\label{completedescr}

% $\tikz\node[circle,draw,inner sep=2pt,outer sep=1pt]{1};$
% $\pscirclebox[fillstyle=solid,fillcolor=red,linecolor=yellow]{\textcolor{white}{1}}$
% $\pscirclebox{\text{A}}$

We now give a complete description of the possible population dynamics. 
Sections \ref{annexeAss1} and \ref{annexeAss3} are dedicated to Assumptions \ref{ass1} and \ref{ass3}, respectively.
In Section \ref{annexeAss2} we explain how the dynamics under Assumptions \ref{asscondinitiale} and \ref{ass2} can be deduced from the dynamics of the process under 
Assumptions \ref{asscondinitiale} and \ref{ass1}.

\subsection{Assumption \ref{ass1}}\label{annexeAss1}

{In Figure \ref{arbresass1} and Table \ref{tableass1}, we assume that the $1$- and $2$-populations survive the first phase. This happens with a probability close to 
$S_{10}S_{20}/(\beta_1 \beta_2)$ (Lemma \ref{lemdureephase1ass1}) and is the case we are interested in, as otherwise there is no clonal interference and we are brought back
to the invasion of one mutant already studied in \cite{champagnat2006microscopic} and \cite{champagnat2011polymorphic}.
The behaviour of the population process after the first phase depends on the relations between the ecological parameters of the different individual types. We describe the different 
conditions which discriminate between different scenari in Figure \ref{arbresass1}, and list them in their chronological order of appearance. We also indicate the phase during which the conditions 
have an impact on the population process behaviour. 

In Table \ref{tableass1} we describe the "final state" of the population processes under the conditions $\ovalbox{A}$ to $\ovalbox{K}$.
For sake of simplicity we call "final state" in this setting the element of $\R_+^3$
\begin{equation}\label{defFS} FS:= \bigcap_{\eps>0} \{ n \in \R_+^3, \exists \beta(\eps),V(\eps)>0, \liminf_{K \to \infty}\P^{(1,2)}(n \in FS(\eps,K,\beta(\eps),V(\eps)))>0 \} ,\end{equation}
when this intersection is non empty, otherwise we call "final state" the long time behaviour of the three dimensional Lotka-Volterra system close to the rescaled population 
process once the three population types have a size of order $K$. In this case we indicate the corresponding class of the dynamical system in the Zeeman representation (Figure \ref{classes1to33}) and we write "$012$" when the invasion fitness 
signs do not allow to discriminate between a cyclical or stable coexistence of the three types of populations.

We have also indicated in Table \ref{tableass1} the durations of the sweep. They have to be understood in the sense of Lemma \ref{lemdureephase1ass1} for instance, 
when the final state $FS$ described in \eqref{defFS} is non empty. They are good approximations of the 
durations with a probability one up to a constant times $\eps$. When $FS$ is empty, these durations correspond to the time needed for the three populations to hit a size of order $K$. 

}

\begin{figure}[h]
  \centering
% \psfrag{S02pos}{$S_{02}>0$}
% \psfrag{P2}{phase $2$}
% \psfrag{P3}{phase $3$}
% \psfrag{P4}{phase $4$}
% \psfrag{P5}{phase $5$}
% \psfrag{P6}{phase $6$}
% \psfrag{A}{$\pscirclebox{\text{A}}$}
% \psfrag{B}{$\pscirclebox{\text{B}}$}
% \psfrag{C}{$\pscirclebox{\text{C}}$}
% \psfrag{D}{$\pscirclebox{\text{D}}$}
% \psfrag{E}{$\pscirclebox{\text{E}}$}
% \psfrag{F}{$\pscirclebox{\text{F}}$}
% \psfrag{G}{$\pscirclebox{\text{G}}$}
% \psfrag{H}{$\pscirclebox{\text{H}}$}
% \psfrag{I}{$\pscirclebox{\text{I}}$}
% \psfrag{J}{$\pscirclebox{\text{J}}$}
% \psfrag{K}{$\pscirclebox{\text{K}}$}
% \psfrag{S02neg}{$S_{02}<0$}
% \psfrag{Q102pos}{$S_{102}>0$}
% \psfrag{Q102neg}{$S_{102}<0$}
% \psfrag{alpha2}{\scriptsize{$\alpha S_{10} < \frac{S_{02}S_{21}}{|S_{12}||S_{01}|}-1$}}
% \psfrag{alpha3}{\scriptsize{$\alpha S_{10} > \frac{S_{02}S_{21}}{|S_{12}||S_{01}|}-1$}}
% \psfrag{alpha1}{\scriptsize{$\alpha < \frac{1}{S_{10}}+\frac{1}{S_{20}}(\frac{S_{21}}{|S_{01}|}-1)$}}
% \psfrag{alpha4}{\scriptsize{$\alpha > \frac{1}{S_{10}}+\frac{1}{S_{20}}(\frac{S_{21}}{|S_{01}|}-1)$}}
% \psfrag{S12neg}{$S_{12}<0$}
% \psfrag{S12pos}{$S_{12}>0$}
% \psfrag{S01neg}{$S_{01}<0$}
% \psfrag{S201neg}{$S_{201}<0$}
% \psfrag{S201pos}{$S_{201}>0$}
% \psfrag{S012neg}{$S_{012}<0$}
% \psfrag{S012pos}{$S_{012}>0$}
% \psfrag{S12neg}{$S_{12}<0$}
% \psfrag{S12pos}{$S_{12}>0$}
% \psfrag{S01pos}{$S_{01}>0$}
% \psfrag{S21pos}{$S_{21}>0$}
% \psfrag{S21neg}{$S_{21}<0$}
  \includegraphics[width=10cm,height=13cm]{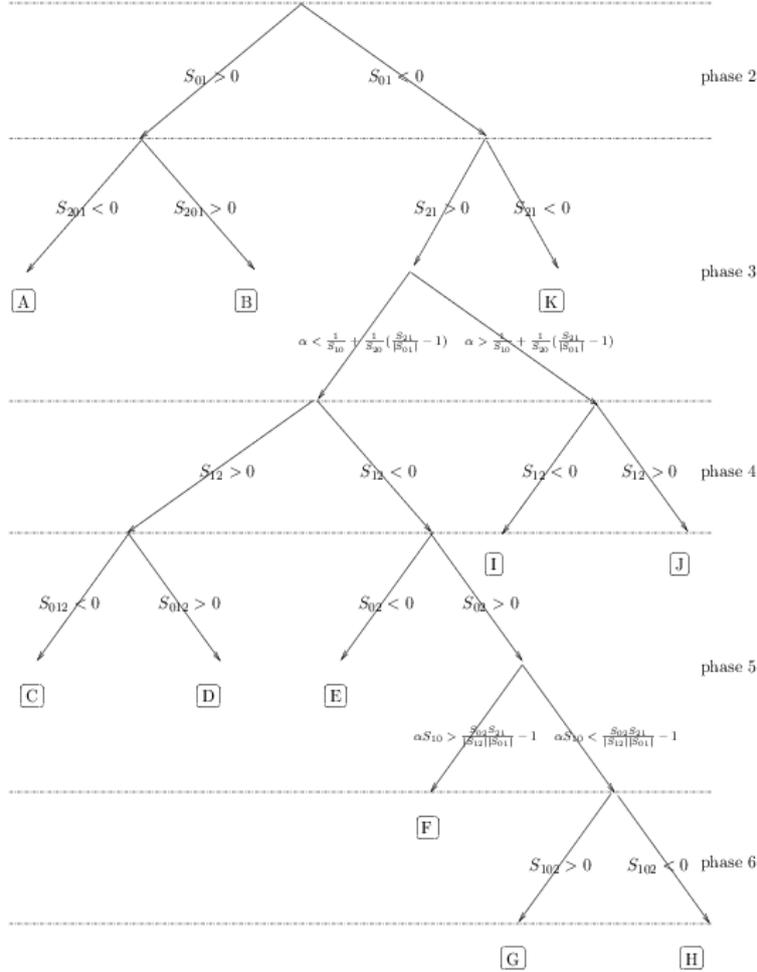}                 
  \caption{Different dynamics under Assumption \ref{ass1}. 
See Section \ref{annexeAss1} for the caption}
  \label{arbresass1}
\end{figure}

\begin{table}
$$\begin{array}{|c|c|c|}
\hline
 \text{Case} & \text{Duration of the sweep:} \log K . & \text{Final state} \\
 \hline
 {\scriptstyle \ovalbox{A}} & {\scriptstyle \frac{1}{S_{10}}} &{\scriptstyle (\bar{n}_{01}^{(0)},\bar{n}_{01}^{(1)},0)} \\
 \hline
{\scriptstyle  \ovalbox{B}} & {\scriptstyle\frac{1}{S_{10}}+\frac{1}{S_{201}}\Big(1-S_{20}\Big( \frac{1}{S_{10}}-\alpha \Big)\Big) } & {\scriptstyle
\begin{array}{l}
  {\scriptstyle S_{21}<0,S_{12}<0,S_{02}<0,\quad (0,0,\bar{n}_2)\ \text{(8)}}\\
 {\scriptstyle S_{21}>0,S_{12}<0,S_{02}<0,\quad 012\ \text{(29)}}\\
 {\scriptstyle S_{21}>0,S_{12}<0,S_{02}>0,S_{102}>0\quad 012\ \text{(31)}}\\
 {\scriptstyle S_{21}>0,S_{12}<0,S_{02}>0,S_{102}<0\quad (\bar{n}_{02}^{(0)},0,\bar{n}_{02}^{(2)})\ \text{(10)}}\\
  {\scriptstyle S_{21}>0,S_{12}>0,S_{02}<0,S_{012}<0\quad (0,\bar{n}_{12}^{(1)},\bar{n}_{12}^{(2)})\ \text{(9)}}\\
 {\scriptstyle S_{12}>0,S_{02}<0, \left\{\begin{array}{l} {\scriptstyle S_{21}<0 \text{ or}}\\ {\scriptstyle S_{21}>0, S_{012}>0}
\end{array}\right.}
\quad  {\scriptstyle\text{012}\ \text{(29 or 31)}}\\
 {\scriptstyle S_{21}>0,S_{12}>0,S_{02}>0,\quad (\bar{n}_{012}^{(0)},\bar{n}_{012}^{(1)},\bar{n}_{012}^{(2)})\ \text{(33)}}
\end{array}}\\
\hline 
{\scriptstyle\ovalbox{C} }&{\scriptstyle\frac{1}{S_{10}}+\frac{1}{S_{21}}\Big( \frac{1}{S_{10}}-\alpha \Big)} & {\scriptstyle(0,\bar{n}_{12}^{(1)},\bar{n}_{12}^{(2)})}\\
\hline 
{\scriptstyle\ovalbox{D}} & {\scriptstyle\frac{1}{S_{10}}+\frac{1}{S_{21}}\Big(1+\frac{|S_{01}|}{S_{012}}\Big)\Big( \frac{1}{S_{10}}-\alpha \Big)}
 & {\scriptstyle
\begin{array}{l}
 {\scriptstyle S_{102}<0,\quad (\bar{n}_{02}^{(0)},0,\bar{n}_{02}^{(2)})\ \text{(10)}}\\
 {\scriptstyle S_{102}>0,\quad \text{012}\ \text{(31)}}\end{array}}\\
\hline 
{\scriptstyle\ovalbox{E}} & {\scriptstyle\frac{1}{S_{10}}+\frac{1}{S_{21}}\Big(1-S_{20}\Big( \frac{1}{S_{10}}-\alpha \Big)\Big)}  &{\scriptstyle (0,0,\bar{n}_{2})}\\
\hline 
{\scriptstyle\ovalbox{F}} & {\scriptstyle\frac{1}{S_{10}}+\frac{1}{S_{21}}\Big(1+\frac{|S_{01}|}{S_{02}}\Big)\Big( \frac{1}{S_{10}}-\alpha \Big)} &{\scriptstyle (\bar{n}_{02}^{(0)},0,\bar{n}_{02}^{(2)})}\\
\hline 
{\scriptstyle\ovalbox{G}} & {\scriptstyle\frac{1}{S_{10}}+\frac{1}{S_{21}}\Big(1+\frac{|S_{01}|}{S_{02}}+\frac{|S_{01}|S|_{12}|}{S_{02}S_{102}}\Big)\Big( \frac{1}{S_{10}}-\alpha \Big)} & {\scriptstyle
\text{012}\ \text{(29)}}\\
\hline 
{\scriptstyle\ovalbox{H}} & {\scriptstyle\frac{1}{S_{10}}+\frac{1}{S_{21}}\Big(1+\frac{|S_{01}|}{S_{02}}\Big)\Big( \frac{1}{S_{10}}-\alpha \Big)} &{\scriptstyle (\bar{n}_{02}^{(0)},0,\bar{n}_{02}^{(2)})}\\
\hline 
{\scriptstyle\ovalbox{I}} & {\scriptstyle\frac{1}{S_{10}}+\frac{1}{S_{21}}\Big(1-S_{20}\Big( \frac{1}{S_{10}}-\alpha \Big)\Big) }  &{\scriptstyle (0,0,\bar{n}_{2})}\\
\hline 
{\scriptstyle\ovalbox{J}} & {\scriptstyle\frac{1}{S_{10}}+\frac{1}{S_{21}}\Big( \frac{1}{S_{10}}-\alpha \Big) } &{\scriptstyle (0,\bar{n}_{12}^{(1)},\bar{n}_{12}^{(2)})}\\
\hline 
{\scriptstyle\ovalbox{K}} & {\scriptstyle\frac{1}{S_{10}}} &{\scriptstyle (0,\bar{n}_{1},0)}\\
\hline 
\end{array}$$
\caption{Possible dynamics under Assumption \ref{ass1}. Encircled letters correspond to these in Figure \ref{arbresass1}.
See Section \ref{annexeAss1} for the caption.}
\label{tableass1}
\end{table}

\subsection{Assumption \ref{ass2}} \label{annexeAss2}

Let us now explain how we deduce the possible dynamics under Assumption \ref{ass2} from the possible dynamics under Assumption \ref{ass1}.
Applying Lemmas \ref{lemphase11} to \ref{lemdureephase1ass2} we get the dynamics represented in Figure \ref{figAss1toAss2} with 
a probability close to $S_{10}S_{20}/(\beta_1\beta_2)$ under Assumptions (\ref{asscondinitiale} and)
\ref{ass1} and \ref{ass2}, respectively. We deduce that to obtain an equivalent 
of Figure \ref{tableass1} it is enough to:
\begin{enumerate}
 \item Interchange the roles of the $1$- and $2$-populations,
 \item Replace $\alpha$ by $-\alpha$ in the conditions of Figure \ref{arbresass1} and in the duration of the sweep
 \item Add $\alpha\log K$ to the duration of the sweep.
\end{enumerate}
We see that we do not get new long time behaviours by comparison with Assumptions \ref{asscondinitiale} and \ref{ass1}.
This is why we did not detail this case.

\begin{figure}[h]
  \centering
  \includegraphics[width=12cm,height=5cm]{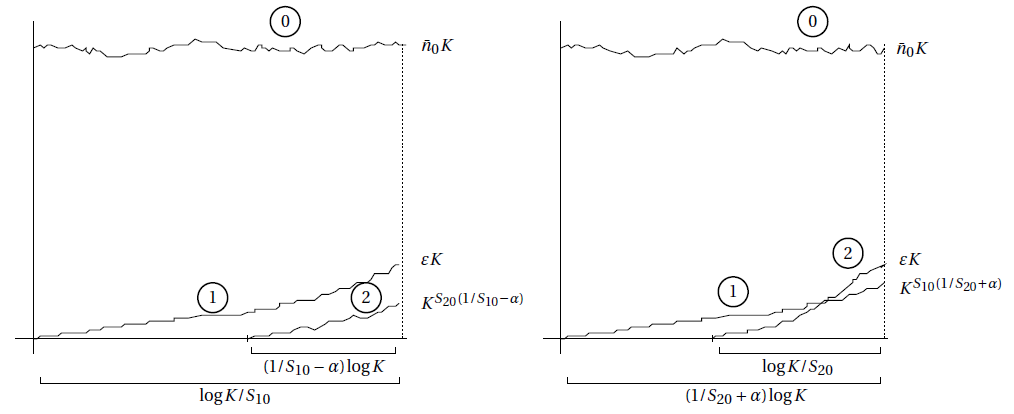}                 
  \caption{From Assumption \ref{ass1} to Assumption \ref{ass2} (see Section \ref{annexeAss2})}
  \label{figAss1toAss2}
\end{figure}

\subsection{Assumption \ref{ass3}}\label{annexeAss3}

Figure \ref{arbresass2} and Table \ref{tableass2} are the analogues of Figure \ref{arbresass1} and Table \ref{tableass1} for Assumption \ref{ass3}.
Here we only assume that the $1$-population survives the first phase (probability close to $S_{10}/\beta_1$ according to Lemma \ref{lemphase11}), 
as the second mutation occurs during the third phase. The captions are the same, except that we add in Table \ref{tableass2} a final 
state "cycles Rock-Paper-Scissors" which corresponds to the class 27 in Zeeman's classification (Proposition \ref{prolizards}).

\begin{figure}[h]
  \centering
  \includegraphics[width=10cm,height=13cm]{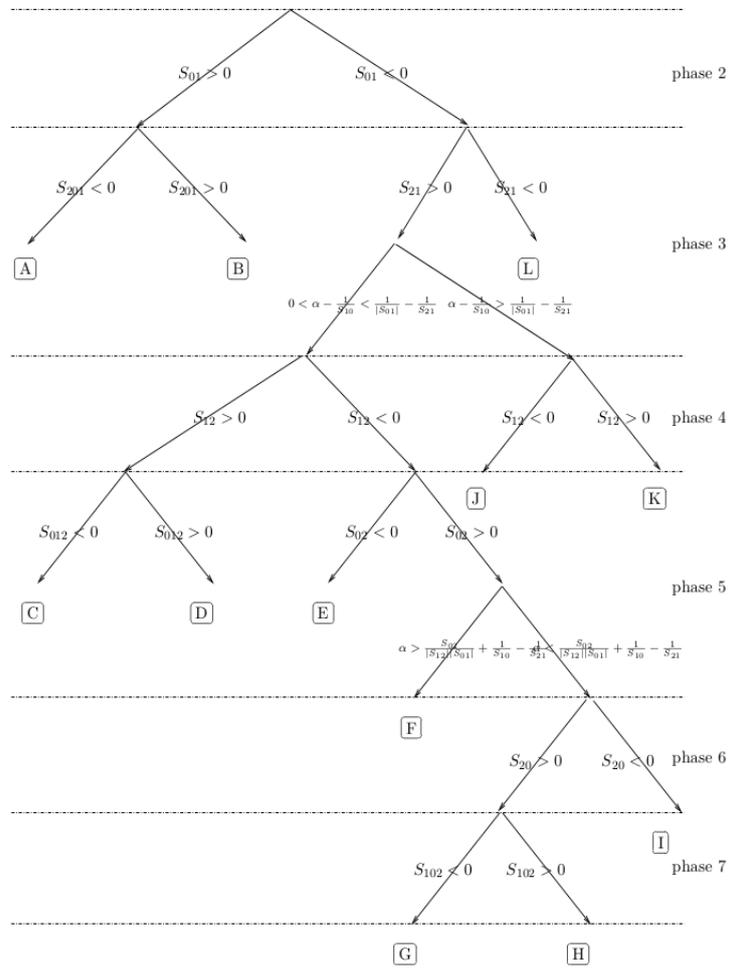}                 
  \caption{Different dynamics under Assumption \ref{ass3}. Same legend as in Figure \ref{arbresass1}}
  \label{arbresass2}
\end{figure}

\begin{table}
$$\begin{array}{|c|c|c|}
\hline
 \text{Case} & \text{Duration of the sweep:} \log K . & \text{Final state} \\
 \hline
  {\scriptstyle \ovalbox{A}} &  {\scriptstyle\frac{1}{S_{10}}} &  {\scriptstyle(\bar{n}_{01}^{(0)},\bar{n}_{01}^{(1)},0)} \\
 \hline
  {\scriptstyle \ovalbox{B}}& \alpha+\frac{1}{S_{201}}  & 
\begin{array}{l}
 {\scriptstyle S_{12}<0,S_{02}<0, S_{20}>0,\quad (0,0,\bar{n}_2)\quad \text{(7 or 8)}}\\
 {\scriptstyle S_{21}>0,S_{12}>0,S_{012}<0,\quad (0,\bar{n}_{12}^{(1)},\bar{n}_{12}^{(2)})\ \text{(9 to 12)}}\\
 {\scriptstyle S_{20}>0,S_{02}>0,S_{102}<0,\quad (\bar{n}_{02}^{(0)},0,\bar{n}_{02}^{(2)})\ \text{(9 to 12)}}\\
 {\scriptstyle S_{12}>0,S_{02}<0, S_{20}>0,\quad \text{012}\ \text{(29 or 31)}}\\
 {\scriptstyle S_{ij}>0,\forall (i,j)\in \{0,1,2\}^2,\quad (\bar{n}_{012}^{(0)},\bar{n}_{012}^{(1)},\bar{n}_{012}^{(2)})\ \text{(33)}}
\end{array}\\
\hline 
 {\scriptstyle \ovalbox{C}} &  {\scriptstyle\alpha+\frac{1}{S_{201}}} &  {\scriptstyle(0,\bar{n}_{12}^{(1)},\bar{n}_{12}^{(2)})}\\
\hline 
 {\scriptstyle \ovalbox{D}} & {\scriptstyle \alpha+\frac{1}{S_{21}}+\frac{|S_{01}|}{S_{012}}\Big( \alpha-\frac{1}{S_{10}}+\frac{1}{S_{21}} \Big)}
 & \begin{array}{l}
 {\scriptstyle S_{20}>0,S_{102}<0,\quad (\bar{n}_{02}^{(0)},0,\bar{n}_{02}^{(2)})\ \text{(10)}}\\
 {\scriptstyle S_{20}>0,S_{102}>0,\quad \text{012}\ \text{(31)}}\\
 {\scriptstyle S_{20}<0,\quad \text{012}\ \text{(29)}}
\end{array}\\
\hline 
 {\scriptstyle \ovalbox{E}} &  {\scriptstyle\alpha+\frac{1}{S_{201}}} &  {\scriptstyle(0,0,\bar{n}_{2})}\\
\hline 
 {\scriptstyle \ovalbox{F}} &  {\scriptstyle \alpha+\frac{1}{S_{21}}+\frac{|S_{01}|}{S_{02}}\Big( \alpha-\frac{1}{S_{10}}+\frac{1}{S_{21}} \Big)} & 
\begin{array}{l}
 {\scriptstyle S_{20}>0,\quad (\bar{n}_{02}^{(0)},0,\bar{n}_{02}^{(2)})}\\
 {\scriptstyle S_{20}<0,\quad (\bar{n}_0,0,0)}
\end{array}\\
\hline 
 {\scriptstyle \ovalbox{G}} &  {\scriptstyle \alpha+\frac{1}{S_{21}}+\frac{|S_{01}|}{S_{02}}\Big( \alpha-\frac{1}{S_{10}}+\frac{1}{S_{21}} \Big) } & 
 {\scriptstyle(\bar{n}_{02}^{(0)},0,\bar{n}_{02}^{(2)})}\\
\hline 
 {\scriptstyle \ovalbox{H}} & 
{\scriptstyle\alpha+\frac{1}{S_{21}}+\frac{|S_{01}|}{S_{02}}\Big(1+\frac{|S_{12}|}{S_{102}}\Big)\Big( \alpha-\frac{1}{S_{10}}+\frac{1}{S_{21}} \Big)}  & 
{\scriptstyle\text{012}\ \text{(29)}}\\
\hline 
 {\scriptstyle \ovalbox{I}} & {\scriptstyle \alpha+\frac{1}{S_{21}}+\frac{|S_{01}|}{S_{02}}\Big(1+\frac{|S_{12}|}{S_{10}}\Big)\Big( \alpha-\frac{1}{S_{10}}+\frac{1}{S_{21}} \Big) } & 
 {\scriptstyle \text{cycles Rock-Paper-Scissors}}\\
\hline 
 {\scriptstyle \ovalbox{J}} &{\scriptstyle \alpha+\frac{1}{S_{201}}}  & {\scriptstyle(0,0,\bar{n}_{2})}\\
\hline 
 {\scriptstyle \ovalbox{K}} & {\scriptstyle\alpha+\frac{1}{S_{201}}} & {\scriptstyle(0,\bar{n}_{12}^{(1)},\bar{n}_{12}^{(2)})}\\
\hline 
 {\scriptstyle \ovalbox{L}} & {\scriptstyle\frac{1}{S_{10}}} & {\scriptstyle(0,\bar{n}_{1},0)}\\
\hline 
\end{array}$$
\caption{Possible dynamics under Assumption \ref{ass3}. Encircled letters correspond to these in Figure \ref{arbresass2}.
See Section \ref{annexeAss3} for the caption.}
\label{tableass2}
\end{table}

\vspace{.5cm}
{\bf Acknowledgements:} {\sl
The authors would like to thank Olivier Tenaillon who suggested this research subject several years ago, and Sylvie Méléard 
for her comments. This work was partially funded by project MANEGE "Modèles Aléatoires en Ecologie,
Génétique et Evolution"
of the French national
research agency ANR-09-BLAN-0215 and Chair "Modélisation Mathémathique et Biodiversité"
of Veolia Environnement - Ecole,Polytechnique - Muséum National d'Histoire Naturelle -
Fondation X and the French national research agency ANR-11-BSV7- 013-03.}

%  \newpage
 \clearpage

\bibliographystyle{abbrv}
% \bibliography{biblio_clonal}

\end{document}